\documentclass[11pt]{article}
\usepackage[margin=1.5cm]{geometry}
\usepackage{graphicx,subcaption,lipsum}
\usepackage{enumitem} 
\usepackage{amsmath,amsfonts, bm,amsthm}
\usepackage{amssymb,enumitem,mathtools,mathrsfs}
\usepackage[ruled,vlined,linesnumbered]{algorithm2e}
\usepackage{graphicx}
\usepackage{hyperref}
\usepackage{hypcap}
\usepackage{tikz}

\usepackage{stackengine,scalerel}
\stackMath

\theoremstyle{plain}
\newtheorem{theorem}{Theorem}[section]
\newtheorem{corollary}[theorem]{Corollary}
\newtheorem{lemma}[theorem]{Lemma}

\newtheorem{proposition}[theorem]{Proposition}

\theoremstyle{definition}

\newtheorem{example}[theorem]{Example}
\newtheorem{remark}[theorem]{Remark}

\makeatletter
\newcommand*{\transpose}{%
  {\mathpalette\@transpose{}}%
}
\newcommand*{\@transpose}[2]{%
  \raisebox{\depth}{$\m@th#1\intercal$}%
}

\title{Characteristic polynomials of $\{\pm 1\}$-matrices modulo a power of $2$}
\author{Gary Greaves and Huu An Phan}
\date{}

\begin{document}
\maketitle
\begin{abstract}
    For a fixed integer $e \geqslant 3$ and $n$ large enough, we show that the number of congruence classes modulo $2^e$ of characteristic polynomials of $n \times n$ symmetric $\{\pm 1\}$-matrices with constant diagonal is equal to 
    \[\begin{cases}
        2^{\binom{e-2}{2}}, & \text{ if $n$ is even}; \\
        2^{\binom{e-2}{2}+1}, & \text{ if $n$ is odd},
    \end{cases}\]
    thereby solving a conjecture of Greaves and Yatsyna from 2019.
    We also show that, for $n$ large enough, the number of congruence classes modulo $2^e$ of characteristic polynomials of $n \times n$ skew-symmetric $\{\pm 1\}$-matrices with constant diagonal is equal to 
    \[\begin{cases}
        2^{\lfloor \frac{e-1}{2} \rfloor\lfloor \frac{e-2}{2} \rfloor}, & \text{ if $n$ is even}; \\
        2^{\lfloor \frac{e-2}{2} \rfloor\lfloor \frac{e-3}{2} \rfloor}, & \text{ if $n$ is odd}.
    \end{cases}\]
    We introduce the concept of a \textit{lift graph/tournament}, which serves as our main tool.
    We also introduce the notion of the \emph{walk polynomial} of a graph, which enables us to show the existence of the requisite lift tournaments.
\end{abstract}

\section{Introduction}

A fundamental problem in discrete random matrix theory is to understand the arithmetic and algebraic properties of random $\{\pm1\}$-matrices.  
Tikhomirov~\cite{T20} proved that the probability that a random $\{\pm1\}$-matrix of order $n$ is singular tends to $2^{-n}$ as $n \to \infty$.  
Eberhard~\cite{eberhard2022characteristic} established that the characteristic polynomial of a random $\{\pm1\}$-matrix is irreducible with high probability, using reductions modulo a prime~$p$ as a key ingredient.  
In this paper we consider the corresponding reduction modulo powers of~$2$.

Characteristic polynomials of $\{\pm1\}$-matrices also arise naturally in combinatorial matrix theory, notably in the study of equiangular lines~\cite{GS24,GSY21,GSY23,GY19}.  
Let
\[
\mathsf S_n = \{\,M \in \textsf{Mat}_n(\{\pm1\}) : M^\transpose = M,\; [M]_{i,i} = 1\ \text{for all } i\,\},
\]
so that $\mathsf S_n$ is the set of symmetric $\{\pm1\}$-matrices with all diagonal entries equal to~$1$.  
The determinant of any $\{\pm1\}$-matrix of order~$n$ is divisible by $2^{n-1}$.  
Greaves and Yatsyna~\cite{GY19} strengthened this observation as follows.

\begin{theorem}[Greaves--Yatsyna~{\cite{GY19}}]
\label{thm:symUB}
Let $e \geqslant 3$. 
Then for all $n \in \mathbb N$,
\[
\left |\{\,\det(xI-S) \bmod 2^e : S \in \mathsf S_n\,\}\right |
   \leqslant
   \begin{cases}
     2^{\binom{e-2}{2}}, & \text{if $n$ is even},\\[3pt]
     2^{\binom{e-2}{2}+1}, & \text{if $n$ is odd}.
   \end{cases}
\]
\end{theorem}

They conjectured \cite[Page 5 and Page 9]{GY19} that this bound is sharp for every fixed $e \geqslant 3$\footnote{It was, in fact, conjectured for all $e \in \mathbb N$, however, when $e \leqslant 2$, the cardinality of $\mathcal C_e(\mathsf S_n)$ is just $1$ for all $n$.} and all sufficiently large~$n$.  
Our first main result confirms this conjecture (Theorem~\ref{thm:m2}).

Analogous behaviour is observed for skew-symmetric matrices.  
Klanderman et~al.~\cite{SeidelTourn24} studied Seidel matrices of tournaments, i.e., skew-symmetric $\{0,\pm1\}$-matrices with nonzero off-diagonal entries, and computed their characteristic polynomial coefficients.  
Their congruence patterns modulo powers of~$2$ mirror those of the symmetric case, leading naturally to a parallel investigation for skew-symmetric $\{\pm1\}$-matrices.  
Our corresponding result is Theorem~\ref{thm:m3}.

\subsection{Main definitions and results}

For $a \in \mathbb Z$ and $e \in \mathbb N$, write $\overline {a}^{(e)}$ for the residue class of $a$ modulo $2^e$.  
Given $\mathsf M_n \subseteq \textsf{Mat}_n(\mathbb Z)$, define
\[
\mathcal C_e(\mathsf M_n)
   = \left \{
   \left (\overline{a_2}^{(e)},\overline{a_3}^{(e)},\dots,\overline{a_e}^{(e)} \right)
   : \det(xI-M) = \sum_{i=0}^n a_i x^{n-i} \text{ for some } M \in \mathsf M_n
   \right \}.
\]

Let $\mathsf U_n$ be the set of $\{\pm1\}$-matrices with all diagonal entries equal to~$1$.  
We call $M \in \mathsf U_n$ \textbf{skew-symmetric} if $(M-I)^\transpose = -(M-I)$, and write $\mathsf S_n$ and $\mathsf T_n$ for the symmetric and skew-symmetric subsets of~$\mathsf U_n$, respectively.

\begin{theorem}
\label{thm:m1}
For each $e \in \mathbb N$ there exists $N$ such that for all $n > N$,
\[
\mathcal C_e(\mathsf U_n)
   = \left\{\, \left (\overline{a_2}^{(e)},\overline{a_3}^{(e)},\dots,\overline{a_e}^{(e)} \right ) : a_i \in \mathbb Z \text{ and }2^{i-1} \text{ divides } a_i \text{ for each } i = 2,\dots,e\,\right\}.
\]
\end{theorem}

Theorem~\ref{thm:m1} is proved as a warm-up result in Section~\ref{sec:proofm1}.

\begin{theorem}
\label{thm:m2}
Let $e \geqslant 3$. Then there exists $N$ such that for all $n > N$,
\[
|\mathcal C_e(\mathsf S_n)|
   =
   \begin{cases}
     2^{\binom{e-2}{2}}, & \text{if $n$ is even},\\[3pt]
     2^{\binom{e-2}{2}+1}, & \text{if $n$ is odd}.
   \end{cases}
\]
\end{theorem}

The proof of Theorem~\ref{thm:m2} consists of Theorem~\ref{thm:m2odd} (for $n$ odd) and Theorem~\ref{thm:m2even} (for $n$ even).

\begin{theorem}
\label{thm:m3}
Let $e \geqslant 3$. Then there exists $N$ such that for all $n > N$,
\[
|\mathcal C_e(\mathsf T_n)|
   =
   \begin{cases}
     2^{\lfloor\frac{e-1}{2}\rfloor\lfloor\frac{e-2}{2}\rfloor}, & \text{if $n$ is even},\\[3pt]
     2^{\lfloor\frac{e-2}{2}\rfloor\lfloor\frac{e-3}{2}\rfloor}, & \text{if $n$ is odd}.
   \end{cases}
\]
\end{theorem}

The proof of Theorem~\ref{thm:m3} consists of Theorem~\ref{thm:m3odd} (for $n$ odd) and Theorem~\ref{thm:m3even} (for $n$ even).

\subsection{Outline of the proofs}

Our arguments are elementary, relying on relations between elementary-symmetric and power-sum polynomials together with combinatorial constructions involving paths, cycles, and almost-transitive tournaments.  
A central tool is the notion of a \emph{lift graph}, introduced in Section~\ref{sec:sumapp}.  
For the symmetric case, the relevant lift graphs can be constructed explicitly from unions of paths and cycles.  
The skew-symmetric case requires a non-explicit construction based on joins of $1$-almost transitive tournaments in the sense of~\cite{fox2021acyclic}.  
We introduce the \emph{walk polynomial} in Section~\ref{sec:wp}, which facilitates the verification of the required congruence conditions for the existence of the relevant lift tournaments.

Our results determine, for each fixed $e$, the precise number of distinct congruence classes of characteristic polynomials of natural families of $\{\pm1\}$-matrices for sufficiently large~$n$.  
Our bounds on the minimal such~$n$ are recursive and, based on empirical evidence, far from optimal.
We leave the problem of optimising these bounds as an avenue for further research.

\subsection{Organisation}

In Section~\ref{sec:pelim}, we outline preliminary definitions and results.
In Section~\ref{sec:summary}, we prove Theorem~\ref{thm:m1} and motivate the notion of a lift graph, which is fundamental to our approach towards the proofs of Theorem~\ref{thm:m2} and Theorem~\ref{thm:m3}.
Lift graphs are defined in Section~\ref{sec:liftgraph} together with constructions of lift graphs of type I and type II.
These constructions are explicit and consist of path and cycle graphs.
Our proof of Theorem~\ref{thm:m2} is presented in Section~\ref{sec:pm2}.
Lift tournaments are defined in Section~\ref{sec:lt} together with constructions of lift tournaments of type I and type II.
In contrast to lift graphs, our constructions of lift tournaments are not explicit and depend on a solution to a Vandermonde-type equation and corresponding joins of almost-transitive tournaments.
Finally, our proof of Theorem~\ref{thm:m3} is presented in Section~\ref{sec:pm3}.

\bigskip

\noindent{\bf Acknowledgements.} GG was supported by the Singapore Ministry of Education Academic Research Fund; grant numbers: RG14/24 (Tier 1) and MOE-T2EP20222-0005 (Tier 2).

\section{Preliminaries}
\label{sec:pelim}

Let $p(x) = c_0\prod_{i=1}^{n}(x-\theta_i) \in \mathbb Q[x]$ be a polynomial.
Denote by $\mathsf c_k(p)$ for the coefficient of $x^{n-k}$ in $p(x)$. 
For $k\in \{0,\dots,n\}$, we define 
the corresponding elementary symmetric polynomial and the power sum polynomial in the zeros of $p(x)$ as
\[
\mathsf e_k(p) :=\sum_{1\leqslant i_1<i_2<\cdots<i_k\leqslant \deg p}\theta_{i_1}\theta_{i_2}\cdots \theta_{i_k} \quad \text{ and } \quad \mathsf p_k(p):=\sum^{\deg p}_{i=1}\theta_i^{k}.\] 

It follows from these definitions and from Newton's identities that 
$\mathsf e_k(p) = (-1)^k\mathsf c_{k}(p)/\mathsf c_0(p)$ and 
\begin{equation}
\label{eqn:power_sum_elementary_symmetric_sum_relation}
            \mathsf p_k(p)=-\frac{1}{\mathsf c_0(p)}\left (\sum_{i=1}^{k-1}\mathsf c_i(p) \mathsf p_{k-i}(p)+ k \mathsf c_k(p)\right ).
        \end{equation}
We further remark here that, for polynomials $p(x)$ and $q(x)$ we have
\begin{equation}
    \label{eqn:powerProdSum}
    \mathsf p_k(p\cdot q)= \mathsf p_k(p)+\mathsf p_k(q).
\end{equation}

For a nonzero integer $a$, the $2$-adic valuation $\nu_2(a)$ of $a$ is the multiplicity of $2$ in the prime factorisation of $a$. 
The 2-adic valuation of $0$ is defined to be $\infty$. 
For an integer $b$ relatively prime to $a$, the $2$-adic valuation $\nu_2(a/b)$ of $a/b$ is defined as $-\nu_2(b)$ if $b$ is even and $\nu_2(a)$ otherwise. 

\begin{lemma}
    \label{lem:ptoe}
    Let $p(x) \in \mathbb Z[x]$ and $m, n \in \mathbb N$ with $m > n$.
    Suppose that
    \[
    \mathsf p_k(p) \equiv 0 \mod {2^{m}} \text{ for each } k \in \{1,\dots,n-1\}\cup \{n+1\}.
    \]
    Then $\mathsf c_0(p)\mathsf p_{n}(p) \equiv -n\mathsf c_{n}(p) \mod{2^{m- \nu_2((n-1)!)}}$ and
    \[
    \mathsf c_k(p) \equiv 0 \mod {2^{m- \nu_2(k!)}} \text{ for each } k \in \{1,\dots,n-1\}\cup \{n+1\}.
    \]
\end{lemma}
\begin{proof}
    We prove by induction on $k$ that
$2^{m-\nu_2(k!)}$ divides $\mathsf c_k(p)$ for each $k \in \{1,\dots,n\}$.
The base case $k=1$ follows immediately since $\mathsf c_1(p) = -\mathsf c_0(p)\mathsf p_1(p)$.
Assume that
$2^{m-\nu_2(j!)}$ divides $\mathsf c_k(p)$ for each $k \in \{1,\dots,j\}$.
Using \eqref{eqn:power_sum_elementary_symmetric_sum_relation}, we have
\[
\mathsf c_0(p)\mathsf p_{j+1}(p)+\sum_{i=1}^{j}\mathsf c_i(p) \mathsf p_{j+1-i}(p)+(j+1) \mathsf c_{j+1}(p) \equiv 0 \mod{2^{m-\nu_2(j!)}}.
\]
Hence $\mathsf c_{j+1}(p) \equiv 0 \mod{2^{m-\nu_2((j+1)!)}}$, as required.

It further follows that $\mathsf c_0(p)\mathsf p_{n}(p) \equiv -n\mathsf c_{n}(p) \mod{2^{m-\nu_2((n-1)!)}}$.
Lastly, 
again using \eqref{eqn:power_sum_elementary_symmetric_sum_relation}, we have
\begin{align*}
    \mathsf c_0(p)\mathsf p_{n+1}(p) &= -\sum_{i=1}^{n}\mathsf c_i(p) \mathsf p_{n+1-i}(p)-(n+1) \mathsf c_{n+1}(p) \\
    &\equiv  \mathsf c_1(p)\mathsf c_n(p) -(n+1) \mathsf c_{n+1}(p) \mod{2^{m-\nu_2((n-1)!)}}.
\end{align*}
Hence, $\mathsf c_{n+1}(p)  \equiv 0 \mod{2^{m-\nu_2((n-1)!)-\nu_2(n+1)}}$, as required.
\end{proof}

Throughout, $J$ denotes the all-ones matrix, whose order can be determined by context.
We denote the characteristic polynomial of a matrix $M$ by $\operatorname{Char}_M(x) := \det(xI-M)$.
We use $M^\transpose$ to denote the transpose of the matrix $M$, and $\mathbf 1$ to denote the all-ones (column) vector where $\mathbf 1 \cdot \mathbf 1^\transpose = J$.
Denote the $(u,v)$-entry of $M$ by $[M]_{u,v}$ and the $v$th entry of a vector $\mathbf v$ by $[\mathbf v]_v$.

We will frequently make use of the next lemma, which is a consequence of the matrix determinant lemma that relates the characteristic polynomial of a matrix $A$ to that of $J-2A$.

\begin{lemma}[{\cite[Lemma 3.1]{GY19}}]
\label{lem:coefficient_link_formula}
Let A be a matrix of order $n$. 
Then, for each $k \in \{0,\dots,n\}$, we have
    \begin{equation*}
\mathsf c_{k}(\operatorname{Char}_{J-2A})=(-2)^k \left (\mathsf c_{k}(\operatorname{Char}_A)+\frac{1}{2}\sum_{i=1}^{k}\mathsf c_{k-i}(\operatorname{Char}_A) \mathbf 1^\transpose A^{i-1}\mathbf 1 \right ).
\end{equation*}
\end{lemma}

The next result shows that the top two (resp.\ three) coefficients of the characteristic polynomial of a matrix in $\mathsf U_n$ (resp.\ $\mathsf S_n$ or $\mathsf T_n$) are determined by $n$.

\begin{lemma}
    \label{lem:first3coeffs}
    Let $M \in \mathsf U_n$.
    Then $\mathsf c_0(\operatorname{Char}_{M}) = 1$, $\mathsf c_{1}(\operatorname{Char}_{M}) = -n$, and $2^{k-1}$ divides $\mathsf c_{k}(\operatorname{Char}_{M})$ for each $k \in \{1,\dots,n\}$.
    Furthermore, if $M \in \mathsf S_n$ then $\mathsf c_{2}(\operatorname{Char}_{M}) = 0$, and if $M \in \mathsf T_n$ then $\mathsf c_{2}(\operatorname{Char}_{M}) = n(n-1)$.
\end{lemma}
\begin{proof}
    Clearly, $A=(J-M)/2$ is a $\{0,1\}$-matrix.
    By Lemma~\ref{lem:coefficient_link_formula}, for each $k \in \{1,\dots,n\}$,
     $2^{k-1}$ divides $\mathsf c_{k}(\operatorname{Char}_{M})$.
     Finally, observe that $\operatorname{tr}(M) = n$ and
     \[
     \operatorname{tr}(M^2) = \begin{cases}
         n^2, & \text{ if } M \in \mathsf S_n; \\
         n(2-n), & \text{ if } M \in \mathsf T_n.
     \end{cases}
     \]
     The lemma then follows from \eqref{eqn:power_sum_elementary_symmetric_sum_relation}.
\end{proof}

If $M \in \mathsf T_n$ we can obtain a further relation between the coefficients of $\operatorname{Char}_M(x)$, which we highlight in the next lemma.

\begin{lemma}
    \label{lem:TcoeffRel}
    Let $M \in \mathsf T_n$.
    Then, for odd $k \in \{1,\dots,n\}$, we have 
        \begin{equation}
    \label{eqn:akeq}
        \mathsf c_k(\operatorname{Char}_M) = -\sum_{i=0}^{k-1}\binom{n-i}{n-k}\mathsf c_i(\operatorname{Char}_M).
    \end{equation}
    Furthermore, if in addition, $n$ is odd, then
    \begin{equation}
    \label{eqn:akeq2}
        (n-k+1)\mathsf c_{k-1}(\operatorname{Char}_M)+\sum_{i=0}^{k-2}\binom{n-i}{n-k}\mathsf c_i(\operatorname{Char}_M) \equiv 0 \mod {2^{k-1}}.
    \end{equation}
\end{lemma}
\begin{proof}
    Let $S = M-I$.
    Then $S^\transpose = -S$.
    Furthermore,
    \[
    \operatorname{Char}_{S}(x) = \det(xI - S) = \det(xI - S^\transpose) = \det(xI + S) = \operatorname{Char}_{-S}(x).
    \]
    It follows that $\mathsf c_k(\operatorname{Char}_S) = 0$ for each odd $k$.
    On the other hand, since  $\operatorname{Char}_{S}(x) = \operatorname{Char}_{M}(x+1)$, we have $\mathsf c_k(\operatorname{Char}_S) = \sum_{i=0}^{k}\binom{n-i}{n-k}\mathsf c_i(\operatorname{Char}_M)$.
    Whence, \eqref{eqn:akeq} follows.
    By Lemma~\ref{lem:first3coeffs}, we have that $2^{k-1}$ divides $\mathsf c_k(\operatorname{Char}_M)$, from which we obtain \eqref{eqn:akeq2}.
\end{proof}

We will have cause to consider both undirected and directed graphs, as well as tournaments.
A directed graph (or \textbf{digraph}) $\Gamma$ is a pair $(X,R)$ where $X$ is a set of \textit{vertices} and $R$ is a subset of $X \times X$ consisting of \textit{arcs}.
The cardinality of $X$ is called the \textbf{order} of $\Gamma$.
If $(i,j) \in R$ then we say that there is an \textbf{arc} from $i$ to $j$ in $\Gamma$. 
Given a digraph $\Gamma = (X,R)$, we denote the set of vertices and arcs by $V(\Gamma) = X$ and $R(\Gamma) = R$, respectively.
The adjacency matrix $A(\Gamma)$ of $\Gamma = (X,R)$ is a $\{0,1\}$-matrix whose rows/columns are indexed by $X$ where $(i,j)$-entry is equal to $1$ if there is an arc from $i$ to $j$ in $\Gamma$.

An \textbf{undirected graph} is a digraph $(X,R)$ where $(i,j) \in R$ if and only if $(j,i) \in R$ and a \textbf{tournament} is a digraph $(X,R)$ where $(i,j) \in R$ if and only if $(j,i) \not \in R$.
A digraph $(X,R)$ is called \textbf{simple} if $(i,i) \not \in R$ for all $i \in X$. 
For the sake of brevity, we refer to an \textit{undirected, simple} graph as merely a \textbf{graph}.
The disjoint union of two digraphs $(X_1,R_1)$ and $(X_2,R_2)$ is defined as $(X_1,R_1) \cup (X_2,R_2) := (X_1 \cup X_2,R_1 \cup R_2)$.
For a positive integer $d$ and a digraph $\Gamma$, by $d\Gamma$, we mean the disjoint union of $d$ copies of $\Gamma$.
If $d > 1$, then $d\Gamma$ cannot be a tournament.
Hence, in the setting of tournaments (see Section~\ref{sec:lt}), we reserve this notation for the ($d$-fold) \emph{join} of a tournament with itself.

Next, we state a result of Harary and Schwenk~\cite[Corollary 5a]{HararySchwenk1979}, which was rediscovered in \cite[Lemma 2.2 and Lemma 2.3]{GY19}.
Let $\varphi(\cdot)$ denote Euler's totient function.

\begin{theorem}
\label{thm:burnside}
    Let $\Gamma$ be an (undirected, simple) graph and $N \geqslant 3$ be an integer.
    Then
        \[
    \sum_{d\,|\,N} \varphi(N/d) \operatorname{tr}(A(\Gamma)^d) \equiv \begin{cases}
        0, & \text{ if $N$ is odd;} \\
        -\frac{N}{2} \mathbf 1^\transpose A(\Gamma)^{N/2} \mathbf 1, & \text{ if $N$ is even}
    \end{cases} \mod {2N}.
        \]
\end{theorem}

A \textbf{walk} of length $k$ in a digraph $\Gamma$ is sequence of vertices $x_0,x_1,\dots,x_k$ such that $(x_i,x_{i+1})$ is an arc in $\Gamma$ for all $i \in \{0,\dots,k-1\}$.
A walk is called \textbf{closed} if $x_0 = x_k$.
Throughout, we will frequently use the basic fact from algebraic graph theory~\cite[Proposition 1.3.1]{bh} that $\left [A(\Gamma)^k\right ]_{u,v}$ is equal to the number of walks in $\Gamma$ of length $k$ from vertex $u$ to vertex $v$.
Clearly, $\mathbf 1^\transpose A(\Gamma)^{0} \mathbf 1$ is equal to the order of $\Gamma$.
Furthermore, we will require the following proposition.

\begin{proposition}
\label{pro:evensumwalks}
Let $\Gamma$ be an (undirected, simple) graph and $k \in \mathbb N \backslash \{0\}$.
Then
    $\mathbf 1^\transpose A(\Gamma)^{k} \mathbf 1$ is even.
\end{proposition}
\begin{proof}
    It suffices to show that $\operatorname{tr}(A(\Gamma)^k) = \mathsf p_k\left ( \operatorname{Char}_{A(\Gamma)}\right)$ is even, which can be done by a simple induction on $k$ using \eqref{eqn:power_sum_elementary_symmetric_sum_relation} together with the fact \cite[Corollary 3.7]{GY19} that $\mathsf c_k\left ( \operatorname{Char}_{A(\Gamma)}\right)$ is even when $k$ is odd.
\end{proof}

\section{Proof of Theorem~\ref{thm:m1} and summary of approach}
\label{sec:summary}

\subsection{Proof of Theorem~\ref{thm:m1}}
\label{sec:proofm1}

We begin by establishing an upper bound on $|\mathcal C_e(\mathsf U_n)|$.

\begin{lemma}
\label{lem:Uub}
    Let $e \in \mathbb N$.
    Then, for all $n \in \mathbb N$, we have
    \[
    |\mathcal C_e(\mathsf U_n)| \leqslant 2^{\binom{e}{2}}.
    \]
\end{lemma}
\begin{proof}
Let $M \in \mathsf U_n$.
We count the number of possible congruence classes of each coefficient $\mathsf c_i(\operatorname{Char}_M)$ modulo $2^e$.
    By Lemma~\ref{lem:first3coeffs}, the coefficient $\mathsf c_0(\operatorname{Char}_M) = 1$, $\mathsf c_1(\operatorname{Char}_M) = -n$, and $2^{j-1}$ divides $\mathsf c_j(\operatorname{Char}_M)$ for each $j \in \{1,\dots,n\}$.
    Thus, for $k \in \{2,\dots,e\}$ there are $2^{e-k+1}$ possible congruence classes for $\mathsf c_k(\operatorname{Char}_M)$ modulo $2^e$.
    For $k \geqslant e+1$, there is just one possible congruence class for $\mathsf c_k(\operatorname{Char}_M)$ modulo $2^e$.
    Hence, there are at most $2 \cdot 2^2 \dots \cdot 2^{e-1}=2^{\binom{e}{2}}$ possible congruence classes for $\operatorname{Char}_{M}(x)$ modulo $2^e$.
\end{proof}

It suffices to prove the corresponding lower bound for $|\mathcal C_e(\mathsf U_n)|$.
We denote by $\vec{P}_n$ the digraph $(X,R)$ with vertex set $X = \{1,\dots,n\}$ where $R = \{(i,i+1) \; : \; i \in \{1,\dots,n-1\} \}$.

\begin{proof}[Proof of Theorem~\ref{thm:m1}]
By Lemma~\ref{lem:Uub}, it suffices to show $|\mathcal C_e(\mathsf U_n)| \geqslant 2^{\binom{e}{2}}$ for $n$ large enough.
For each $f \in \{1,\dots,e\}$, define the digraph $\Gamma_f = (2^e-1)\vec{P}_{f-1} \cup \vec{P}_{f-1+2^e}$.
    It is clear that $\operatorname{Char}_{A(\Gamma_f)}(x) = x^{2^e f}$.
    Furthermore, for each $k \in \mathbb N$, we have
    \[
    \mathbf 1^\transpose A(\Gamma_f)^k \mathbf 1 \equiv \mathbf 1^\transpose A(\vec{P}_{f-1+2^e})^k \mathbf 1 - \mathbf 1^\transpose A(\vec{P}_{f-1})^k \mathbf 1 \mod{2^e}.
    \]
    Observe that 
    \[
    \mathbf 1^\transpose A(\vec{P}_{n})^k \mathbf 1 = \begin{cases}
        n - k, & \text{ if $k \leqslant n$}; \\
        0, & \text{ if $k \geqslant n$}.
    \end{cases}
    \]
    Hence, for each $k \in \{0,1,\dots,e\}$,
\begin{equation}
\label{eqn:terms}
    \mathbf 1^\transpose A(\Gamma_f)^k \mathbf 1 \equiv \begin{cases}
        0  & \text{ if } k<f;\\
        f-k-1  & \text{ if } k \geqslant f
    \end{cases} \mod{2^{e}}.
\end{equation}

    We choose $N_e=2^{2e}\cdot \binom{e+1}{2}$. 
    It suffices to show that for $n>N_e$ and $(c_2,c_3,\dots,c_{e}) \in \{1,\dots,2^e\}^{e-1}$, there exists an $n$-vertex digraph $\Gamma$ such that, for each $k \in \{2,3,\dots,e\}$, we have $\mathsf c_k(\operatorname{Char}_{J-2A(\Gamma)}) \equiv 2^{k-1}c_k \mod {2^e}$.
    
    Take $(d_1,d_2,\dots,d_{e-1}) \in \{1,\dots,2^e\}^{e-1}$ such that, for each $i \in \{2,\dots,e\}$, we have
    $$d_{i-1}+2d_{i-2}+3d_{i-3}+\dots+(i-1)d_1\equiv (-1)^{i-1}c_{i}\mod{2^e}.$$
    Let
    \[
    \Gamma=d_1\Gamma_1\cup d_2\Gamma_2\cup \dots \cup d_{e-1}\Gamma_{e-1} \cup (n-2^e\sum_{f=1}^{e-1} f\cdot d_f)\vec{P}_1.
    \]
    Clearly, we have $\operatorname{Char}_{A(\Gamma)}(x)=x^n$.
    By Lemma~\ref{lem:coefficient_link_formula}, for each $k \in \{2,\dots,e\}$,
        \begin{align}
        \label{eqn:arIII}
            \mathsf c_k(\operatorname{Char}_{J-2A(\Gamma)}) &= 2^{k-1}(-1)^{k} \sum_{i=1}^{e-1} d_i \mathbf 1^\transpose A(\Gamma_i)^{k-1} \mathbf 1.
        \end{align}
    Using \eqref{eqn:terms}, \eqref{eqn:arIII} becomes
           \begin{align*}
          \mathsf c_k(\operatorname{Char}_{J-2A(\Gamma)})  &\equiv 2^{k-1}(-1)^k\sum_{i=1}^{k-1}d_i(i-k) \equiv 2^{k-1}c_k\mod{2^e},
        \end{align*}
        as required.
\end{proof}

Note that, in the proof of Theorem~\ref{thm:m1}, for each tuple $\left ( \overline{a_2}^{(e)},\overline{a_3}^{(e)}, \dots, \overline{a_{e}}^{(e)} \right ) \in \mathcal C_e(\mathsf U_n)$ we were able to directly find a digraph $\Gamma$ such that $\mathsf c_k(\operatorname{Char}_{J-2A(\Gamma)}) \equiv a_k$ for each $k \in \{2,\dots,e\}$.
This worked since we can find digraphs with at least one arc whose characteristic polynomial is a monomial.
For the cases of $\mathsf S_n$ and $\mathsf T_n$, however, this does not work.
Instead, our approach is to take a tuple in $\mathcal C_e(\mathsf M_n)$ and ``lift'' it to tuples in $\mathcal C_{e+1}(\mathsf M_{n^\prime})$ for some $n^\prime > n$.

\subsection{Summary of our approach}
\label{sec:sumapp}

In this section, we give an overview of our approach for establishing the requisite lower bounds.
The techniques required to establish the upper bounds have been covered in \cite{GY19}.

To illustrate the idea, we provide a simple example.
\begin{example}
\label{ex:P2P2}
    Suppose $e,n \in \mathbb N$ with $e \geqslant 3$, $n$ odd, and $\left ( \overline{a_2}^{(e)},\overline{a_3}^{(e)},\dots,\overline{a_e}^{(e)} \right ) \in \mathcal C_e(\mathsf S_n)$. 
    Let $\Gamma$ be an $n$-vertex graph such that $c_k = \mathsf c_k(\operatorname{Char}_{J-2A(\Gamma)}) \equiv a_k \mod 2^e$ for each $k \in \{2,\dots,e\}$.
    It is straightforward to show that, for some integer $c_{e+1}$,
    \[
\left ( \overline{0}^{(e+1)},\overline{c_3}^{(e+1)},\dots,\overline{c_e}^{(e+1)},\overline{c_{e+1}}^{(e+1)} \right ) \in \mathcal C_{e+1}(\mathsf S_{n+2^e}).
\]
Indeed, one can use the matrix $J-2A(\Gamma \cup 2^e\vec{P}_1)$.

    Now, consider the graph $\Lambda = 2^{e-2}P_2$ where $P_2 = (\{1,2\},\{(1,2),(2,1)\})$.
    Observe that 
    $\mathbf 1^\transpose A(\Lambda)^k \mathbf 1 = 2^{e-1}$ for all $k \in \mathbb N \cup \{0\}$.
    Furthermore, since $\operatorname{Char}_{A(\Lambda)}(x) = (x^2-1)^{2^{e-2}}$, it follows that, for each $k \in \{1,\dots,e\}$, we have
    \begin{equation}
    \label{eqn:f2}
        \mathsf c_k(\operatorname{Char}_{A(\Lambda)}) \equiv \begin{cases}
        0 , &  \text{ if } k \ne 2;\\
        2^{e-2}, & \text{ if } k=2
    \end{cases} \mod{2^{e+1-k}}.
    \end{equation}

    Let $\Gamma^\prime = \Gamma \cup \Lambda$.
    Then clearly, for each  $k \in \mathbb N \cup \{0\}$, we have $\mathbf 1^\transpose A(\Gamma^\prime)^k \mathbf 1 = \mathbf 1^\transpose A(\Gamma)^k \mathbf 1 + 2^{e-1}$.
    Furthermore, $\operatorname{Char}_{\Gamma^\prime}(x)=\operatorname{Char}_{\Gamma}(x)\operatorname{Char}_{\Lambda}(x)$, whence, for each $k \in \{1,\dots,e\}$,
    \begin{align*}
    \mathsf c_k(\operatorname{Char}_{A(\Gamma^\prime)}) &=\sum_{j=0}^k\mathsf c_j(\operatorname{Char}_{A(\Lambda)})\mathsf c_{k-j}(\operatorname{Char}_{A(\Gamma)}).
    \end{align*}
Combining with \eqref{eqn:f2}, we find that, for each $k \in \{1,\dots,e\}$,
    \[
    (-2)^k\mathsf c_k(\operatorname{Char}_{A(\Gamma^\prime)}) \equiv \begin{cases}
         (-2)^k\mathsf c_k(\operatorname{Char}_{A(\Gamma)}) , & \text{if } k\ne 2;\\
         (-2)^k\mathsf c_k(\operatorname{Char}_{A(\Gamma)})+2^e, & \text{if } k=2
    \end{cases} \mod{2^{e+1}}.
    \]
    
Now, using Lemma~\ref{lem:coefficient_link_formula} and Proposition~\ref{pro:evensumwalks}, for each $k \in \{3,\dots,e\}$, we have
\begin{align*}
\mathsf c_{k}(\operatorname{Char}_{J-2A(\Gamma^\prime)})&=(-2)^k \left (\mathsf c_{k}(\operatorname{Char}_{A(\Gamma^\prime)})+\frac{1}{2}\sum_{i=1}^{k}\mathsf c_{k-i}(\operatorname{Char}_{A(\Gamma^\prime)}) \mathbf 1^\transpose A(\Gamma^\prime)^{i-1}\mathbf 1 \right ) \\
&\equiv   (-2)^k \left (\mathsf c_{k}(\operatorname{Char}_{A(\Gamma)})+\frac{1}{2}\sum_{i=1}^{k}\mathsf c_{k-i}(\operatorname{Char}_{A(\Gamma)}) \mathbf 1^\transpose A(\Gamma)^{i-1}\mathbf 1 \right ) + 2^{e}\mathbf 1^\transpose A(\Gamma)^{k-3}\mathbf 1 \mod{2^{e+1}} \\
    &\equiv \begin{cases}
        \mathsf c_{k}(\operatorname{Char}_{J-2A(\Gamma)}), &  \text{ if } k \ne 3;\\
        \mathsf c_{k}(\operatorname{Char}_{J-2A(\Gamma)})+2^{e}, & \text{ if } k=3
    \end{cases} \mod{2^{e+1}}.
\end{align*}
As a consequence, we see that for some integer $c_{e+1}^\prime$,
\[
\left ( \overline{0}^{(e+1)},\overline{c_3+2^e}^{(e+1)},\dots,\overline{c_e}^{(e+1)}, \overline{c_{e+1}^\prime}^{(e+1)} \right ) \in \mathcal C_{e+1}(\mathsf S_{n+2^e}).
\]
\end{example}

Example~\ref{ex:P2P2} demonstrates that there exist certain graphs that one can use to lift tuples from $\mathcal C_e(\mathsf M_n)$ to $\mathcal C_{e+1}(\mathsf M_{n+2^e})$ in a controlled manner in the sense that, after lifting to modulo $2^{e+1}$, only the congruence class of the third coefficient was shifted by $2^e$.
One consequence of Example~\ref{ex:P2P2} is that the cardinality of $\mathcal C_{e+1}(\mathsf M_{n+2^e})$ is at least twice that of $\mathcal C_{e}(\mathsf M_{n})$.
Below, we will develop this approach to show that, for some $n^\prime > n$, the cardinality of $\mathcal C_{e+1}(\mathsf M_{n^\prime})$ must be a larger multiple that of $\mathcal C_{e}(\mathsf M_{n})$
The graph $\Lambda$ in Example~\ref{ex:P2P2} motivates the definition of a ``lift graph", which we define below.

Our approach is inductive.
Consider a subset $\mathsf M_n \subset \mathsf{Mat}_n(\mathbb Z)$.
For each tuple $\left ( \overline{a_2}^{(e)},\overline{a_3}^{(e)}, \dots, \overline{a_{e}}^{(e)} \right ) \in \mathcal C_e(\mathsf M_n)$, there must exist $M \in \mathsf M_n$ such that $\mathsf c_k(\operatorname{Char}_M) \equiv a_k \mod {2^e}$ for all $k \in \{2,3,\dots,e\}$.
Our strategy is to ``lift'' the matrix $M$ to a larger corresponding matrix $M^\prime$ where $\mathsf c_k(\operatorname{Char}_{M^\prime}) \equiv a_k + \varepsilon_k 2^e \mod {2^{e+1}}$ for each $k \in \{2,3,\dots,e\}$ where $\varepsilon_k \in \{0,1\}$.
The possible values for $\varepsilon_k$ will depend on the set of matrices $\mathsf M_n$ that we are considering, i.e., for some $k \in \{2,3,\dots,e\}$ the corresponding value of $\varepsilon_k$ may always equal $0$.

In order to find this ``lifting" matrix $M^\prime$, we introduce the notion of a lift graph (resp.\ lift tournament).
We will demonstrate that lift graphs and lift tournaments exhibit the desired lifting effect on the coefficients of the resulting characteristic polynomials, and we will also need to show that such digraphs exist.

\section{Lift graphs}
\label{sec:liftgraph}

In this section, we define two types of lift graphs and provide constructions of families of these graphs.
First, we define lift graphs of type I and type II.

\subsection{Definitions}

For positive integers $e$ and $f$, we call a graph $\Gamma$ an $(e,f)$-\textbf{lift graph of type I} if 
\begin{enumerate}
    \item[(L1a)] $2^{e-2}$ divides $\mathbf 1^\transpose A(\Gamma)^k \mathbf 1$ for each $k \in \mathbb N \cup \{0\}$;
    \item[(L1b)] for each $k \in \{1,\dots,e-1\}$,
    \[
\mathsf c_k(\operatorname{Char}_{A(\Gamma)}) \equiv \begin{cases}
        0 , &  \text{ if } k \ne f;\\
        2^{e-f-1}, & \text{ if } k=f
    \end{cases} \mod{2^{e-k}}.
\]
\end{enumerate}

\begin{remark}
    \label{rem:K2K2}
    We remark that the graph $2P_2$ (as in Example~\ref{ex:P2P2}) is a $(4,2)$-lift graph of type I.
    Indeed,  $\operatorname{Char}_{A(2P_2)} = x^4-2x^2+1$ and $\mathbf 1^\transpose A(2P_2)^k \mathbf 1 = 2\mathbf 1^\transpose A(P_2)^k \mathbf 1 =4$ for all $k \in \mathbb N \cup \{0\}$.
\end{remark}

For $e \in \mathbb N$, we call a graph $\Gamma$ an $e$-\textbf{lift graph of type II} if 
\begin{enumerate}
    \item[(L2a)] for each $k \in \{0,1,\dots,e\}$,
\[
\mathbf 1^\transpose A(\Gamma)^k \mathbf 1 \equiv \begin{cases}
        0,  & \text{ if } k \ne e;\\
        2,   & \text{ if } k=e
    \end{cases} \mod{2^{e+2-k}};
\]
\item[(L2b)] $2^{e+2-k}$ divides $\mathsf c_{k}(\operatorname{Char}_{A(\Gamma)})$ for each $k \in \{1,2,\dots,e+1\}$.
\end{enumerate}

\subsection{Lift graphs of type I}

Remark~\ref{rem:K2K2} shows that there exist $(4,2)$-lift graphs of type I.
Before we can provide a construction of $(f+2,f)$-lift graphs of type I for $f \geqslant 3$, we establish a formula for the trace of powers of the adjacency matrix of a cycle graph $C_n$, which is defined as
$C_n = \left (\mathbb Z/n\mathbb Z, \{ (i,i+1) \; : \; i \in \mathbb Z/n\mathbb Z \} \cup \{ (i,i-1) \; : \; i \in \mathbb Z/n\mathbb Z \} \right )$.

\begin{proposition}[cf.~{\cite[Proposition 2.3.]{dsouza2013combinatorial}}]\label{pro:trace:k_not_equal_n}
Let $n \geqslant 3$ be an integer.
Then, for each $k \in \{1,2,\dots,n+1\}$, we have
    \[
    \operatorname{tr}(A(C_n)^k) =
    \begin{cases}
    0, & \text{if } n \ne k \text{ is odd}; \\
    n{\binom{k}{\frac{k}{2}}} , & \text{if } n \ne k \text{ is even}; \\
    2n, & \text{if } n = k \text{ is odd}; \\
    n{\binom{n}{\frac{n}{2}}} +2n, & \text{if } n = k \text{ is even}.
    \end{cases}
    \]
\end{proposition}
\begin{proof}
    We count the number of closed walks \(\mathbf{x}=x_0x_1\cdots x_{k-1}\). 
    For $i\in \{0,1,\dots, k-1\}$, $x_{i+1}=x_i+1$ or $x_{i+1}=x_i-1$. Let $m$ be the number of indices $i\in \{0,1,\dots, k-1\}$ such that $x_{i+1}=x_i+1$, then $k-m$ is the number of indices $i$ such that $x_{i+1}=x_i-1$. Since \(\mathbf{x}\) is a closed walk, it follows that $n$ divides $m-(k-m)$.
    Furthermore, if $k \leqslant n$ then $|2m-k| \leqslant n$, which implies $2m = k$ or $n = k$ and $m \in \{0,n\}$.
    Otherwise, if $k = n+1$ then $|2m-k| \leqslant n+1$.
    But, since $2m-k \equiv n+1 \mod 2$, we must have $2m = k$.

    For the case \(m=\frac{k}{2}\), the number of closed walks is $n\cdot {\binom{k}{\frac{k}{2}}}$, that is, $n$ ways to choose $x_0$ and \({\binom{k}{\frac{k}{2}}}\) ways to choose $m$ indices from $\{0,1,\dots,k-1\}$.
\end{proof}

Now we provide a construction of $(f+2,f)$-lift graphs of type I for each integer $f \geqslant 3$.

\begin{lemma}
\label{lem:cycle_unionLiftI}
    Let $f \geqslant 3$ be an integer.
    Then $C_f\cup C_{2^{f+1}(f+1)!-f}$ is an $(f+2,f)$-lift graph of type I.
\end{lemma}
\begin{proof}
    Let $\Gamma = C_f\cup C_{2^{f+1}(f+1)!-f}$, $n = 2^{f+1}(f+1)!$, and $A = A(\Gamma)$.
    Since $\Gamma$ is $2$-regular, we have
    \[
    \mathbf 1^\transpose A^k \mathbf 1 = n2^k.
    \]
    Thus, $\Gamma$ satisfies (L1a).
    Since $n-f>f+1$, using Proposition~\ref{pro:trace:k_not_equal_n}, it follows that, for $k\leqslant  f+1$, we have
    \[
    \operatorname{tr}(A^k) = \operatorname{tr}(A(C_f)^k )+\operatorname{tr}(A(C_{n-f})^k) \equiv
    \begin{cases}
        0, & \text{if } k \ne f;\\
        2f, & \text{if } k = f
    \end{cases} \mod n.
    \]
Applying Lemma~\ref{lem:ptoe}, we obtain that, for each $k \in \{1,\dots, f+1\}$,
    \[\mathsf c_k(\operatorname{Char}_{A}) \equiv
    \begin{cases}
        0 , & \text{if } k \ne f;\\
        -2, & \text{if } k = f
    \end{cases} \mod 2^{f+1}.
    \]
    Hence, $\Gamma$ satisfies (L1b).
\end{proof}

The next lemma allows us to use the two constructions (Remark~\ref{rem:K2K2} and Lemma~\ref{lem:cycle_unionLiftI}) above to generate $(e,f)$-lift graphs of type I for all integers $e$ and $f$ such that $e \geqslant f+2 \geqslant 4$.

\begin{lemma}\label{lem:inductiveLiftConst}
    Suppose $\Gamma$ is an $(e,f)$-lift graph of type I.
    Then $2\Gamma$ is an $(e+1,f)$-lift graph of type I.
\end{lemma}
\begin{proof}
    For each $k\in \mathbb N$, we have  $\mathbf 1^\transpose A(2\Gamma)^k \mathbf 1 = 2 \mathbf 1^\transpose A(\Gamma)^k \mathbf 1 \equiv 0 \mod {2^{e-1}}$.
    Hence, we obtain (L1a).
    Since $\operatorname{Char}_{2\Gamma}(x) = \operatorname{Char}_{\Gamma}^2(x)$, we can deduce that $\mathsf c_k(\operatorname{Char}_{2\Gamma}) \equiv 2 \mathsf c_k(\operatorname{Char}_{\Gamma}) \mod {2^{e+1-k}}$ for each $k \in \{1,\dots,e-1\}$, from which (L1b) follows.
\end{proof}

\subsection{Lift graphs of type II}

In this section, we provide a construction of $e$-lift graphs on type II.
For this, we will use unions of path graphs.
The path graph $P_n$ is defined as 
\[
P_n=\left (\{0,1,\dots,n-1\},\{(i,i+1) \; : \; i \in \{0,\dots,n-2\}\} \cup \{(i,i-1) \; : \; i \in \{1,\dots,n-1\}\}\right ).
\]
For a nonnegative integer $c$, a vertex $v \in V(P_n)$ is called $c$\textbf{-central} if the minimum distance from $v$ to a leaf vertex is $c$.
The two leaf vertices of $P_n$ are precisely its $0$-central vertices.

Denote by $W_n(v,k)$ the set of walks $\textbf{x}=x_0,x_1,\dots, x_{k}$ in $P_n$ where $x_0 = v$.
    Each walk in $W_n(v,k)$ can be identified with a unique binary string $(b_1,\dots, b_{k}) \in \{0,1\}^k$, with $b_i = 1$ if $x_{i}=x_{i-1}+1$ and $b_i = 0$ otherwise. 
    Denote by $\mathsf w_{n,v,k} : W_n(v,k) \to \{0,1\}^k$ be the corresponding injective map.

    We will require a few results that count the number of certain walks in path graphs.

\begin{proposition}\label{pro:free_walk_of_path_graph}
    Let $n \geqslant 2$ be an integer.
    Suppose $v \in V(P_n)$ is $c$-central.
    Then
    \begin{itemize}
        \item[(i)] for each nonnegative integer $k \leqslant c$,
    \[
    \left [ A(P_n)^k \mathbf 1 \right ]_v = 2^k;
    \]
    \item[(ii)] for each nonnegative integer $k$ satisfying $k \leqslant 2c$,
          \[
    \left [A(P_n)^k\right ]_{v,v}= \begin{cases}
        0,  & \text{ if } k \text{ odd},\\
        {\binom{k}{\frac{k}{2}}}, & \text{ if } k \text{ even.}
    \end{cases}
    \]
    \end{itemize}
\end{proposition}
\begin{proof}
    Since $k \leqslant c$, we have $\mathsf w_{n,v,k}(W_n(v,k)) = \{0,1\}^k$.
    Whence, we obtain (i).

    Now consider the subset $W^\prime(v,k) \subset W(v,k)$ of closed walks.
    Since $x_k = x_0$, the image $\mathsf w_{n,v,k}(\mathbf x)$ must have the same number of 1s as 0s for each $\mathbf x \in W^\prime(v,k)$.
    Thus, if $k$ is odd then $|W^\prime(v,k)| = 0$.
    Otherwise, since $k \leqslant 2c$, we find that $|W^\prime(v,k)| = \binom{k}{k/2}$, which establishes (ii).
\end{proof}

\begin{proposition}
\label{pro:walk_of_odd_path_graph}
    Let $m,n  \in \mathbb N$ with $n \geqslant 2$, let $v \in V(P_{2n-1})$ and  $w \in V(P_{2n-1+2^m})$ both be $c$-central vertices.
        \begin{enumerate}
        \item[(i)] If $c = n-1$ then
        \[
        \begin{aligned}[c]
        \left [ A(P_{2n-1})^{n} \mathbf 1 \right ]_v &= 2^n - 2 \\
         \left [A(P_{2n-1})^{2n}\right ]_{v,v} &= \binom{2n}{n} - 2
        \end{aligned}
        \quad \text{and } \quad 
        \begin{aligned}[c]
        \left [ A(P_{2n-1+2^m})^{n} \mathbf 1 \right ]_w &= 2^n - 1;
          \\
         \left [ A(P_{2n-1+2^m})^{2n} \right ]_{w,w} &= \binom{2n}{n} - 1.
        \end{aligned}
        \]
        \item[(ii)] If \( c < n-1\) and \(k\leqslant n\) then
        \[
        \left [ A(P_{2n-1})^{k} \mathbf 1 \right ]_v =  \left [ A(P_{2n-1+2^m})^{k} \mathbf 1 \right ]_w.
        \]
        \item[(iii)] If \( c \leqslant n-1\) and \(k \leqslant 2n-1\) then
        \[
        \left [ A(P_{2n-1})^k\right ]_{v,v} = \left [A(P_{2n-1+2^m})^k\right ]_{w,w}.
        \]
        \item[(iv)]
        If \( c < n-1\) then
        \[
        \left [A(P_{2n-1})^{2n}\right ]_{v,v} = \left [A(P_{2n-1+2^m})^{2n}\right ]_{w,w}.
        \]
    \end{enumerate}
\end{proposition}
\begin{proof}
        Clearly, we have $\mathsf w_{2n-1,v,n}(W_{2n-1}(v,n)) = \{0,1\}^n \backslash \{(0,\dots,0),(1,\dots,1)\}$.
        In a similar manner, one observes $\mathsf w_{2n-1+2^m,w,n}(W_{2n-1+2^m}(w,n)) = \{0,1\}^n \backslash \{(0,\dots,0)\}$ or $\{0,1\}^n \backslash \{(1,\dots,1)\}$.
        Let $W_{2n-1}^\prime(v,k) \subset W_{2n-1}(v,k)$ and $W^\prime_{2n-1+2^m}(w,k) \subset W_{2n-1+2^m}(w,k)$ be the subsets of closed walks.
        Denote by $B(n)$ the subset of $\{0,1\}^{2n}$ consisting of elements that have precisely $n$ entries equal to $1$.
        Since we are considering closed walks, $\mathsf w_{2n-1,v,2n}(W^\prime_{2n-1}(v,2n))$ and $\mathsf w_{2n-1+2^m,w,2n}(W^\prime_{2n-1+2^m}(w,2n))$ are both subsets of $B(n)$.
        Furthermore, $\mathsf w_{2n-1,v,2n}(W^\prime_{2n-1}(v,2n)) = B(n)\backslash \{ (0,\dots,0,1,\dots,1),(1,\dots,1,0,\dots,0)\}$ and, similarly,
        \[
        \mathsf w_{2n-1+2^m,w,2n}(W^\prime_{2n-1+2^m}(w,2n)) \in \left \{ B(n)\backslash \{(1,\dots,1,0,\dots,0)\}, B(n)\backslash \{ (0,\dots,0,1,\dots,1)\} \right \}.
        \]
        Thus, we obtain (i).
        
        Let $\phi : V(P_{2n-1}) \to V(P_{2n-1+2^m})$ be an injective graph homomorphism where $\phi(v) = w$.
        We will apply $\phi$ to a set of walks of $P_{2n-1}$ by applying $\phi$ to each vertex of each walk to obtain a set of walks in $P_{2n-1+2^m}$.
        It is straightforward to check that, if $c < n-1$ and $k \leqslant n$ then $\phi(W_{2n-1}(v,k)) = W_{2n-1+2^m}(w,k)$, which yields (ii).
        Similarly, it follows that $\phi(W_{2n-1}^\prime(v,k)) = W^\prime_{2n-1+2^m}(w,k)$, for $c < n-1$ and $k \leqslant 2n$ and for $c = n-1$ and $k < 2n$.
        Whence we have (iii) and (iv).
\end{proof}

The proof of the next proposition can be obtained from that of Proposition~\ref{pro:walk_of_odd_path_graph} \emph{mutatis mutandis}.

\begin{proposition}\label{pro:walk_of_even_path_graph}
    Let $m,n  \in \mathbb N$ with $n \geqslant 2$, let $v \in V(P_{2n})$ and  $w \in V(P_{2n+2^m})$ both be $c$-central vertices.
    \begin{enumerate}
    \item[(i)] If $c=n-1$ then
        \[
        \left [ A(P_{2n})^{n+1} \mathbf 1 \right ]_{v}+1 = \left [ A(P_{2n+2^m})^{n+1} \mathbf 1 \right ]_{w}.
        \]
        \item[(ii)] If \( c \leqslant n-1\) and \(k \leqslant n+1\) then
        \[
        \left [ A(P_{2n})^k \mathbf 1 \right ]_{v} = \left [ A(P_{2n+2^m})^k \mathbf 1 \right ]_{w}.
        \] 
        \item[(iii)] If \( c \leqslant n-1\) and \(k \leqslant 2n+1\) then
        \[
        \left [A(P_{2n})^k\right ]_{v,v} = \left [A(P_{2n+2^m})^k\right ]_{w,w}.
        \]
    \end{enumerate}
\end{proposition}

Now we are ready to provide a construction of lift graphs of type II.

\begin{lemma}\label{shift_graph_using_path}
Let $e \geqslant 3$ be an integer and $m=e+1+\nu_2(e!)$.
Then $P_{e-1}\cup (2^m-1)P_{e-1+2^m}$ is an $e$-lift graph of type II.
\end{lemma}
\begin{proof}
Let $\Gamma = P_{e-1}\cup (2^m-1)P_{e-1+2^m}$.
Suppose $A$, $A_1$, and $A_2$ are the adjacency matrices of $\Gamma$, $P_{e-1}$, and $P_{e-1+2^m}$, respectively.
    
    First, suppose that $e$ is even.
    Let $e = 2f$ for some integer $f \geqslant 2$.
    For $k \leqslant e$, we have
    \begin{align}
        \mathbf 1^\transpose A^k \mathbf 1 &= \mathbf 1^\transpose A_1^k \mathbf 1 + (2^m-1)\mathbf 1^\transpose A_2^k \mathbf 1 \nonumber \\ \nonumber
        &\equiv \mathbf 1^\transpose A_1^k \mathbf 1 -\mathbf 1^\transpose A_2^k \mathbf 1 \mod {2^m} \\
        &\equiv  \sum^{e-2}_{v=0}\left [A_1^{\lfloor k/2 \rfloor} \mathbf 1\right ]_v \left [A_1^{\lceil k/2 \rceil} \mathbf 1\right ]_v - \sum^{e-2+2^m}_{v=0}\left [A_2^{\lfloor k/2 \rfloor} \mathbf 1\right ]_v \left [A_2^{\lceil k/2 \rceil} \mathbf 1\right ]_v \mod {2^m}. \label{eqn:1}
    \end{align}

    Note that, using the reflective symmetry of a path graph, for each $v \in \{0,\dots,f-2\}$ and for all $l \in \mathbb N$, we have $\left [A_1^{l} \mathbf 1\right ]_v = \left [A_1^{l} \mathbf 1\right ]_{e-2-v}$ and $\left [A_2^{l} \mathbf 1\right ]_v = \left [A_2^{l} \mathbf 1\right ]_{e-2+2^m-v}$.
    Furthermore, by Proposition~\ref{pro:walk_of_odd_path_graph} (ii), for each $l \in \{0,1,\dots,f\}$ and each $v \in \{0,1,\dots,f-2\}$, we have 
    $\left [A_1^{l} \mathbf 1\right ]_v = \left [A_2^{l} \mathbf 1\right ]_v$.
    Hence, \eqref{eqn:1} becomes
    \begin{equation}
            \mathbf 1^\transpose A^k \mathbf 1 \equiv  \left [A_1^{\lfloor k/2 \rfloor} \mathbf 1\right ]_{f-1} \left [A_1^{\lceil k/2 \rceil} \mathbf 1\right ]_{f-1} - \sum^{f-1+2^m}_{v=f-1}\left [A_2^{\lfloor k/2 \rfloor} \mathbf 1\right ]_v \left [A_2^{\lceil k/2 \rceil} \mathbf 1\right ]_v \mod {2^m}. \label{eqn:1Ak1}
    \end{equation}
    Combining \eqref{eqn:1Ak1} with Proposition~\ref{pro:free_walk_of_path_graph}, for $k \in \{0,1,\dots,e-2\}$, we obtain $\mathbf 1^\transpose A^k \mathbf 1 \equiv 2^k - (2^m+1)2^k \equiv 0 \mod {2^m}$.
    Using Proposition~\ref{pro:free_walk_of_path_graph} and Proposition~\ref{pro:walk_of_odd_path_graph}, if $k = e-1$ then \eqref{eqn:1Ak1} becomes
    \begin{align*}
        \mathbf 1^\transpose A^{e-1} \mathbf 1 &\equiv  2^{f-1} (2^f-2) - 2\cdot 2^{f-1} (2^f-1) - \sum^{f-2+2^m}_{v=f}\left [A_2^{f-1} \mathbf 1\right ]_v \left [A_2^{f} \mathbf 1\right ]_v \mod {2^m} \nonumber \\
        &\equiv - \sum^{f-2+2^m}_{v=f}\left [A_2^{f-1} \mathbf 1\right ]_v \left [A_2^{f} \mathbf 1\right ]_v \mod {8},
    \end{align*}
    which, by Proposition~\ref{pro:free_walk_of_path_graph}, is congruent to $-(2^m-1)2^{f-1}2^f$ modulo $8$.
    Hence, $\mathbf 1^\transpose A^{e-1} \mathbf 1 \equiv 0 \mod 8$.
    Similarly, using Proposition~\ref{pro:free_walk_of_path_graph} and Proposition~\ref{pro:walk_of_odd_path_graph}, if $k = e$ then \eqref{eqn:1Ak1} becomes
    \begin{align*}
        \mathbf 1^\transpose A^{e} \mathbf 1 &\equiv  (2^f-2)^2 - 2(2^f-1)^2 - \sum^{f-2+2^m}_{v=f}\left [A_2^{f} \mathbf 1\right ]_v \left [A_2^{f} \mathbf 1\right ]_v \mod {2^m} \nonumber \\
        &\equiv 2 - \sum^{f-2+2^m}_{v=f}\left [A_2^{f} \mathbf 1\right ]_v \left [A_2^{f} \mathbf 1\right ]_v \mod {4},
    \end{align*}
    which, by Proposition~\ref{pro:free_walk_of_path_graph}, is congruent to $2-(2^m-1)2^{2f}$ modulo $4$.
    Hence, $\mathbf 1^\transpose A^{e} \mathbf 1 \equiv 2 \mod 4$.
    We have established that $\Gamma$ satisfies (L2a).

    It remains to show that $\Gamma$ satisfies (L2b), that is, $2^{e+2-k}$ divides $\mathsf c_k(\operatorname{Char}_A)$ for each $k \in \{1,2,\dots,e+1\}$.
    Observe that
    \begin{align}
        \operatorname{tr}(A^k) &=  \operatorname{tr}(A_1^k) +  (2^m-1)\operatorname{tr}(A_2^k) \nonumber \\
        &\equiv \operatorname{tr}(A_1^k) -\operatorname{tr}(A_2^k) \mod {2^m} \nonumber \\
        &\equiv \sum^{e-2}_{v=0} \left [A_1^k\right ]_{v,v} - \sum^{e-2+2^m}_{v=0}\left [A_2^k \right ]_{v,v} \mod {2^m} \label{eqn:2}.
    \end{align}
    Again, using the reflective symmetry of a path graph, for all $l \in \mathbb N$, we have $\left [A_1^l \right ]_{v,v}= \left [A_1^l \right ]_{e-2-v,e-2-v}$ and $\left [A_2^l \right ]_{v,v}= \left [A_2^l\right ]_{e-2+2^m-v,e-2+2^m-v}$ for each $v \in \{0,\dots,f-2\}$.
    Furthermore, by Proposition~\ref{pro:walk_of_odd_path_graph} (iii), for each $l \in \{0,\dots,e-1\}$ and $v \in \{0,\dots,f-2\}$, we have 
    $\left [A_1^l\right ]_{v,v}= \left [A_2^l\right ]_{v,v}$.
    Hence, \eqref{eqn:2} becomes
    \begin{align}
        \operatorname{tr}(A^k)  
        &\equiv \left [A_1^k\right ]_{f-1,f-1} - \sum^{f-1+2^m}_{v=f-1}\left [A_2^k\right ]_{v,v} \mod {2^m} \label{eqn:trAk}.
    \end{align}
    Suppose $k \leqslant e-2$ and apply Proposition~\ref{pro:free_walk_of_path_graph} (ii) to \eqref{eqn:trAk} to obtain
    \[
    \operatorname{tr}(A^k) \equiv \left [A_1^k\right ]_{f-1,f-1} - (2^m+1)\left [A_2^k\right ]_{f-1,f-1} \equiv \left [A_1^k\right ]_{f-1,f-1} - \left [A_2^k\right ]_{f-1,f-1} \mod {2^m}.
    \]
    Hence, by Proposition~\ref{pro:walk_of_odd_path_graph} (iii), we obtain $\operatorname{tr}(A^k) \equiv 0 \mod {2^m}$.
    Since $e$ is even, it is obvious that $\operatorname{tr}(A^{e-1}) = \operatorname{tr}(A^{e+1}) = 0$.

For $k = e$, using Proposition~\ref{pro:free_walk_of_path_graph} (ii) and Proposition~\ref{pro:walk_of_odd_path_graph} (iv), \eqref{eqn:trAk} becomes
\begin{align*}
    \operatorname{tr}(A^{e})  
        &= \binom{2f}{f} - 2 - 2\left (\binom{2f}{f} - 1 \right ) - (2^m-1)\binom{2f}{f} \equiv 0 \mod {2^m}.
\end{align*}
 Now we can apply Lemma~\ref{lem:ptoe} to deduce $\mathsf c_k(\operatorname{Char}_A) \equiv 0 \mod{2^{e+1}}$ for each $k \in \{1,\dots,e+1\}$.

    The case when $e$ is odd follows similarly,  using Proposition~\ref{pro:walk_of_even_path_graph} in place of Proposition~\ref{pro:walk_of_odd_path_graph}.
    For the convenience of the reader, we give some details for the proof that $\mathsf c_{e+1}(\operatorname{Char}_A)$ is even.
    Apply Theorem~\ref{thm:burnside}, which states
        \[
    \sum_{d\,|\,e+1}\varphi \left (\frac{e+1}{d} \right )\operatorname{tr}(A^d)+f\mathbf 1^\transpose A^f \mathbf 1\equiv 0 \mod{2(e+1)}.
        \]
    Next, using $\mathbf 1^\transpose A^f\mathbf 1 \equiv 0 \mod {2^m}$ and $\operatorname{tr}(A^d) \equiv 0 \mod {2^m}$ for all $d<e+1$, one can deduce that $\mathsf p_{e+1}(\operatorname{Char}_A) = \operatorname{tr}(A^{e+1})\equiv 0 \mod{2^{\nu_2(e+1)+1}}$. 
    Finally, apply \eqref{eqn:power_sum_elementary_symmetric_sum_relation} to obtain $-(e+1)\mathsf c_{e+1}(\operatorname{Char}_A) \equiv 0 \mod{2^{\nu_2(e+1)+1}}$, whence $\mathsf c_{e+1}(\operatorname{Char}_A)$ is even, as required.
    \end{proof}

\section{Proof of Theorem~\ref{thm:m2}}
\label{sec:pm2}

In this section, we prove Theorem~\ref{thm:m2}.
We begin with an inclusion result on $\mathcal C_e(\mathsf S_n)$.

\begin{proposition}
\label{pro:lifting_vertices}
    For $n, e\in \mathbb{N}$, we have $\mathcal C_e(\mathsf S_n) \subset \mathcal C_e(\mathsf S_{n+2^{e-2}})$.
\end{proposition}
\begin{proof}
    Let $ \mathbf a = \left ( \overline{a_2}^{(e)},\overline{a_3}^{(e)},\dots,\overline{a_e}^{(e)}\right) \in \mathcal C_e(\mathsf S_n)$ and 
    let $\Gamma$ be an $n$-vertex graph such that $\mathsf c_i(\operatorname{Char}_{J-2A(\Gamma)}) \equiv a_i \mod{2^e}$ for each $i \in \{2,\dots,e\}$.
    Let $\Gamma^\prime = \Gamma \cup 2^{e-2}P_1$.
    Clearly, $\mathbf 1^\transpose A(\Gamma^\prime)^0 \mathbf 1 = \mathbf 1^\transpose A(\Gamma)^0 \mathbf 1 + 2^{e-2}$ and $\mathbf 1^\transpose A(\Gamma^\prime)^i \mathbf 1 = \mathbf 1^\transpose A(\Gamma)^i \mathbf 1$ for each $i \in \mathbb N$.

    Since $\operatorname{Char}_{A(\Gamma^\prime)}(x) = \operatorname{Char}_{A(\Gamma)}(x)x^{2^{e-2}}$, we have $\mathsf c_{i}(\operatorname{Char}_{A(\Gamma^\prime)})=\mathsf c_{i}(\operatorname{Char}_{A(\Gamma)})$ for each $i \in \{1,2,\dots,n\}$.
    By Lemma~\ref{lem:first3coeffs}, we have $\mathsf c_2(\operatorname{Char}_{J-2A(\Gamma^\prime)}) = \mathsf c_2(\operatorname{Char}_{J-2A(\Gamma)}) = 0$.
    By Lemma~\ref{lem:coefficient_link_formula}, for each $k \in \{3,4,\dots, e\}$, we have
\begin{align*}
\mathsf c_{k}(\operatorname{Char}_{J-2A(\Gamma^\prime)})
    &= (-2)^k\mathsf c_{k}(\operatorname{Char}_{A(\Gamma^\prime)}) + (-1)^k\sum_{i=1}^{k}\mathsf c_{k-i}(\operatorname{Char}_{A(\Gamma^\prime)})\cdot2^{k-1} \cdot \mathbf 1^\transpose A(\Gamma^\prime)^{i-1} \mathbf 1\\
    &\equiv (-2)^k\mathsf c_{k}(\operatorname{Char}_{A(\Gamma)}) + (-1)^k\sum_{i=1}^{k}\mathsf c_{k-i}(\operatorname{Char}_{A(\Gamma)})\cdot2^{k-1} \cdot \mathbf 1^\transpose A(\Gamma)^{i-1} \mathbf 1 \mod{2^e}\\
    &\equiv \mathsf c_{k}(\operatorname{Char}_{J-2A(\Gamma)}) \mod{2^e}.
\end{align*}
Hence $\mathbf a \in \mathcal C_e(\mathsf S_{n+2^{e-2}})$.
\end{proof}

The main utility of lift graphs of type I is the following lemma.

\begin{lemma}\label{lem:shift_graph}
Let $\Gamma$ be a graph and $\Lambda$ be an $(e,f)$-lift graph of type I.
Then, for each $k \in \{1,\dots,e-1\}$,
    \[
    \mathsf c_k(\operatorname{Char}_{A(\Gamma \cup \Lambda)}) \equiv \begin{cases}
         \mathsf c_k(\operatorname{Char}_{A(\Gamma)}) , & \text{if } k\ne f;\\
         \mathsf c_k(\operatorname{Char}_{A(\Gamma)})+2^{e-f-1}, & \text{if } k=f
    \end{cases} \mod{2^{e-k}}.
    \]
    Moreover, $\mathbf 1^\transpose A(\Gamma\cup \Lambda)^k \mathbf 1\equiv \mathbf 1^\transpose A(\Gamma)^k \mathbf 1 \mod{2^{e-2}}$ for all $k \in \mathbb N$.
\end{lemma}
\begin{proof}
    Since $2^{e-2}$ divides $\mathbf 1^\transpose A(\Lambda)^k \mathbf 1$ for each $k \in \mathbb N$, we have 
    \[
    \mathbf 1^\transpose A(\Gamma \cup \Lambda)^k \mathbf 1 = \mathbf 1^\transpose A(\Gamma)^k \mathbf 1+\mathbf 1^\transpose A(\Lambda)^k \mathbf 1 \equiv \mathbf 1^\transpose A(\Gamma)^k \mathbf 1 \mod{2^{e-2}}
    \]
     for each $k \in \mathbb N$.
    We have $\operatorname{Char}_{\Gamma\cup \Lambda}(x)=\operatorname{Char}_{\Gamma}(x)\operatorname{Char}_{\Lambda}(x)$, whence 
    \begin{align*}
    \mathsf c_k(\operatorname{Char}_{A(\Gamma \cup \Lambda)}) &=\sum_{j=0}^k\mathsf c_j(\operatorname{Char}_{A(\Lambda)})\mathsf c_{k-j}(\operatorname{Char}_{A(\Gamma)}) \\
        &= \mathsf c_{k}(\operatorname{Char}_{A(\Gamma)})+\mathsf c_k(\operatorname{Char}_{A(\Lambda)})+\sum_{j=1}^{k-1}\mathsf c_j(\operatorname{Char}_{A(\Lambda)})\mathsf c_{k-j}(\operatorname{Char}_{A(\Gamma)})
    \end{align*}
    for each $k \in \{1,2,\dots,e-1\}$.
    Since $2^{e-1-j}$ divides $\mathsf c_j(\operatorname{Char}_{A(\Lambda)})$ for each  $j \in \{1,2,\dots,e-1\}$, we have 
    \[
    \mathsf c_k(\operatorname{Char}_{A(\Gamma \cup \Lambda)}) \equiv \mathsf c_k(\operatorname{Char}_{A(\Gamma)})+\mathsf c_k(\operatorname{Char}_{A(\Lambda)}) \mod{2^{e-k}}
    \] 
    from which our conclusion follows.    
\end{proof}

Now we can prove that the cardinality of $\mathcal C_{e}(\mathsf S_N)$ is at least $2^{e-3}|\mathcal C_{e-1}(\mathsf S_n)|$ for $e \geqslant 4$ and $n$ and $N$ are both odd, with $N$ large enough relative to $n$.

\begin{lemma}
\label{lem:2em3below}
    Let $n \in \mathbb N$ be odd and $e \geqslant 4$ be an integer.
    Suppose $N = n+\sum_{f=2}^{e-2}(e-f-1)n_f$, where $n_f$ is the order of an $(e,f)$-lift graph of type I.
    Then $\left |\mathcal C_e \left (\mathsf S_N\right ) \right | \geqslant 2^{e-3} \left |\mathcal C_{e-1} \left (\mathsf S_n\right ) \right |$.
\end{lemma}
\begin{proof}
By Remark~\ref{rem:K2K2}, Lemma~\ref{lem:cycle_unionLiftI}, and Lemma~\ref{lem:inductiveLiftConst}, for each $f \in \{2,\dots,e-2\}$, there exists an $(e,f)$-lift graph $\Lambda_f$ of type I.
    Denote by $n_f$ the order of each graph $\Lambda_f$.
    Let $(c_3,\dots,c_{e-1}) \in \{1,\dots,2^e\}^{e-3}$ such that 
    \[
    \left ( \overline {0}^{(e-1)},\overline {c_3}^{(e-1)},\dots,\overline {c_{e-2}}^{(e-1)},\overline {c_{e-1}}^{(e-1)}\right) \in \mathcal C_{e-1}(\mathsf S_{n}).
    \]
    Let $\Gamma$ be an $n$-vertex graph such that 
    $\mathsf c_i(\operatorname{Char}_{J-2A(\Gamma)}) \equiv c_i \mod {2^{e-1}}$ for each $i \in \{3,\dots,e-1\}$.
    Hence, there must exist $(d_3,\dots,d_{e-1}) \in \{0,1\}^{e-3}$ such that $c_i \equiv \mathsf c_i(\operatorname{Char}_{J-2A(\Gamma)}) + 2^{e-1}d_i \mod{2^e}$.
    
    Let 
    \[
    \Gamma^\prime = \Gamma \cup d_{e-1} \Lambda_{e-2} \cup (d_{e-1}+d_{e-2})\Lambda_{e-3} \cup \dots \cup  (d_{e-1}+d_{e-2} + \dots + d_3) \Lambda_2
    \]
    and let $n^\prime$ be the order of $\Gamma^\prime$.
    Then $n^\prime \leqslant n +\sum_{f=2}^{e-2}(e-f-1)n_f =: N$ and  (L1a) implies that $N\equiv n^\prime \equiv n \mod{2^{e-2}}$.
    By Lemma \ref{lem:shift_graph}, for each $k \in \{1,\dots,e-1\}$, we have $\mathbf 1^\transpose A(\Gamma^\prime)^k \mathbf 1 \equiv \mathbf 1^\transpose A(\Gamma)^k \mathbf 1 \mod{2^{e-2}}$ and, by Lemma~\ref{lem:shift_graph},
      \begin{equation*}
      \label{eqn:bk}
           \mathsf c_k(\operatorname{Char}_{A(\Gamma^\prime)}) \equiv \mathsf c_k(\operatorname{Char}_{A(\Gamma)})+(d_{e-1}+\dots+d_{k+1})2^{e-k-1} \mod{2^{e-k}}.
      \end{equation*}

        Hence, for each $j \in \{3,\dots, e-1\}$ and each $k \in \{1,\dots,e-1\}$, we have
        \begin{equation}
        \label{eqn:2bk}
            (-2)^j \mathsf c_k(\operatorname{Char}_{A(\Gamma^\prime)}) \equiv 
        \begin{cases}
            (-2)^j \mathsf c_k(\operatorname{Char}_{A(\Gamma)}), & \text{ if } k<j,\\
            (-2)^k \mathsf c_k(\operatorname{Char}_{A(\Gamma)}) +2^{e-1}(d_{e-1}+\dots+d_{k+1}), & \text{ if } k=j
        \end{cases} \mod{2^{e}}.
        \end{equation}

Now we can deduce that $\mathsf c_k(\operatorname{Char}_{J-2A(\Gamma^\prime)})  \equiv c_k \mod{2^e}$ for each $k \in \{3,\dots, e-1\}$.
Indeed, using \eqref{eqn:2bk}, Lemma~\ref{lem:coefficient_link_formula}, and Proposition~\ref{pro:evensumwalks}, we find that
\begin{align*}
\mathsf c_k(\operatorname{Char}_{J-2A(\Gamma^\prime)}) &= (-2)^k\left ( \mathsf c_k(\operatorname{Char}_{A(\Gamma^\prime)}) +\frac{1}{2}\sum_{i=1}^{k}\mathsf c_{k-i}(\operatorname{Char}_{A(\Gamma^\prime)}) \cdot \mathbf 1^\transpose A(\Gamma^\prime)^{i-1} \mathbf 1 \right ) \\
&\equiv \mathsf c_{k}(\operatorname{Char}_{J-2A(\Gamma)})+2^{e-1}(d_{e-1}+\dots+d_{k+1})+2^{e-1}(d_{e-1}+\dots+d_{k})  \mod{2^e}\\
    &\equiv \mathsf c_{k}(\operatorname{Char}_{J-2A(\Gamma)})+d_k2^{e-1}  \mod{2^e}\\
    &\equiv c_k\mod{2^e}.
\end{align*}

Therefore, for some $c_e$, we have $\left ( \overline {0}^{(e)},\overline {c_3}^{(e)},\overline {c_4}^{(e)},\dots,\overline {c_{e-1}}^{(e)},\overline {c_e}^{(e)}\right)\in \mathcal C_{e}(\mathsf S_{n^\prime})$, which is a subset of $\mathcal C_{e}(\mathsf S_N)$ by Proposition~\ref{pro:lifting_vertices} together with the fact that $N\equiv n^\prime \equiv n \mod{2^{e-2}}$. 
\end{proof}

Using Lemma~\ref{lem:2em3below}, we can prove the odd case of Theorem~\ref{thm:m2} with a simple induction.

\begin{theorem}
\label{thm:m2odd}
    For each integer $e\geqslant 3$,
    there exists an integer $N_e$ such that $|\mathcal C_{e}(\mathsf S_n)| = 2^{\binom{e-2}{2}+1}$ for each odd integer $n \geqslant N_e$. 
\end{theorem}
\begin{proof}
By Theorem~\ref{thm:symUB}, it suffices to show that $|\mathcal C_{e}(\mathsf S_n)| \geqslant 2^{\binom{e-2}{2}+1}$.
       We proceed by induction on $e$.
    For $e = 3$, we can take $N_3 = 3$ and $\{([0]_3,[0]_3),([0]_3,[4]_3)\} \subseteq \mathcal C_{3}(\mathsf S_n)$ for each odd $n \geqslant N_3$.
    Indeed, we have the graphs $\Gamma_1 = nP_1$ and $\Gamma_2 = P_2 \cup (n-2)P_1$.
    For the former $\mathsf c_{3}(\operatorname{Char}_{J-2A(\Gamma_1)}) \equiv 0 \mod 8$ and the latter $\mathsf c_{3}(\operatorname{Char}_{J-2A(\Gamma_1)})  \equiv 4 \mod 8$.

    Now we can assume $e \geqslant 4$ for the inductive step.
    Take $N_{e}=N_{e-1}+\sum_{f=2}^{e-2}(e-f-1)n_f$ where $n_f$ is the order of an $(e,f)$-lift graph of type I.

    Suppose $n = N_e + 2k$ for some nonnegative integer $k$.
    Then, by the inductive hypothesis, $|\mathcal C_{e-1}(\mathsf S_{N_{e-1}+2k})| = 2^{\binom{e-3}{2}+1}$.
    By Lemma~\ref{lem:2em3below}, we have $|\mathcal C_{e}(\mathsf S_n)| \geqslant 2^{e-3}2^{\binom{e-3}{2}+1} = 2^{\binom{e-2}{2}+1}$.
\end{proof}

The main utility of lift graphs of type II is the following lemma.

\begin{lemma}\label{lem:shift_graphII}
    Let $e$ be a positive integer, 
    $\Gamma$ be a graph,
    and $\Lambda$ be an $e$-lift graph of type II. 
    Then, for each $k \in \{1,2,\dots, e+1\}$, we have $\mathsf c_k(\operatorname{Char}_{A(\Gamma \cup \Lambda)}) \equiv \mathsf c_k(\operatorname{Char}_{A(\Gamma)}) \mod{2^{e+2-k}}$, and for each $k \in \{1,2,\dots, e\}$,
\begin{align*}
    \mathbf 1^\transpose{A(\Gamma \cup \Lambda)^k} \mathbf 1 &\equiv
    \begin{cases}
         \mathbf 1^\transpose {A(\Gamma)^k}\mathbf 1, & \text{if } k<e; \\
        \mathbf 1^\transpose {A(\Gamma)^e}\mathbf 1+2, & \text{if } k= e
    \end{cases} \mod{2^{e+2-k}}.
\end{align*}
\end{lemma}
\begin{proof}
    Since $\operatorname{Char}_{\Gamma \cup \Lambda}(x)=\operatorname{Char}_{\Gamma}(x)\operatorname{Char}_{\Lambda}(x)$, we have, for each $k \in \{1,2,\dots, e+1\}$,
    \[
    \mathsf c_k(\operatorname{Char}_{A(\Gamma \cup \Lambda)})=\sum_{j=0}^k\mathsf c_j(\operatorname{Char}_{A(\Gamma)})\mathsf c_{k-j}(\operatorname{Char}_{A(\Lambda)}) = \mathsf c_k(\operatorname{Char}_{A(\Gamma)}) + \sum_{j=0}^{k-1}\mathsf c_j(\operatorname{Char}_{A(\Gamma)})\mathsf c_{k-j}(\operatorname{Char}_{A(\Lambda)}).
    \]
    Since, by (L2b), $2^{e+2-k+j}$ divides $\mathsf c_{k-j}(\operatorname{Char}_{A(\Lambda)})$ for each $j \in \{k-e,\dots,k-1\}$, we obtain 
    \[
    \mathsf c_k(\operatorname{Char}_{A(\Gamma \cup \Lambda)}) \equiv \mathsf c_k(\operatorname{Char}_{A(\Gamma)}) \mod{2^{e+2-k}}.
    \]
    Lastly, for each $k \in \mathbb N$, we have $\mathbf 1^\transpose {A(\Gamma \cup \Lambda)^k}\mathbf 1 = \mathbf 1^\transpose {A(\Gamma)^k}\mathbf 1 + \mathbf 1^\transpose {A(\Lambda)^k}\mathbf 1$.
    The rest of the lemma follows from (L2a).
\end{proof}

Now we can prove that the cardinality of $\mathcal C_{e}(\mathsf S_N)$ is at least $2^{e-3}|\mathcal C_{e-1}(\mathsf S_n)|$ for $e \geqslant 5$ and $n$ and $N$ are both even, with $N$ large enough relative to $n$.

\begin{lemma}
\label{lem:2em3belowEven}
    Let $n \in \mathbb N$ be even and $e \geqslant 5$ be an integer.
    Suppose $N = n+\sum_{f=3}^{e-2}n_f+m$, where $n_f$ is the order of an $(e,f)$-lift graph of type I and $m$ is the order of an $(e-2)$-lift graph of type II.
    Then $\left |\mathcal C_{e}\left (\mathsf S_N \right ) \right | \geqslant 2^{e-3}|\mathcal C_{e-1}(\mathsf S_n)|$.
\end{lemma}
\begin{proof}
    By Lemma~\ref{lem:cycle_unionLiftI} and Lemma~\ref{lem:inductiveLiftConst}, for each $f \in \{3,\dots,e-2\}$, there exists an $(e,f)$-lift graph $\Lambda_f$ of type I.
    By Lemma~\ref{shift_graph_using_path}, there exists an $(e-2)$-lift graph $\Lambda_{e-1}$ of type II.
    Denote by $n_f$ the order of each graph $\Lambda_f$ and by $m$ the order of $\Lambda_{e-1}$.
    Let $(c_3,\dots,c_{e-1}) \in \{1,\dots,2^e\}^{e-3}$ such that 
    \[
    \left ( \overline {0}^{(e-1)},\overline {c_3}^{(e-1)},\dots,\overline {c_{e-1}}^{(e-1)}\right) \in \mathcal C_{e-1}(\mathsf S_{n}).
    \]
    Let $\Gamma$ be an $n$-vertex graph such that 
    $\mathsf c_i(\operatorname{Char}_{J-2A(\Gamma)}) \equiv c_i \mod {2^{e-1}}$ for each $i \in \{3,\dots,e-1\}$.
    Hence, there must exist $(d_3,\dots,d_{e-1}) \in \{0,1\}^{e-3}$ such that $c_i \equiv \mathsf c_i(\operatorname{Char}_{J-2A(\Gamma)}) + 2^{e-1}d_i \mod{2^e}$.
    
    Let 
    \[
    \Gamma^\prime = \Gamma \cup d_{3} \Lambda_{3} \cup d_4\Lambda_4 \cup \dots \cup  d_{e-2} \Lambda_{e-2} \cup d_{e-1} \Lambda_{e-1}
    \]
    and let $n^\prime$ be the order of $\Gamma^\prime$.
    Then $n^\prime \leqslant n +\sum_{f=3}^{e-2} n_f + m =: N$ and Properties (L1a) and (L2a) imply that  $N \equiv n^\prime \equiv n \mod{2^{e-2}}$.
    By Lemma \ref{lem:shift_graph} and Lemma~\ref{lem:shift_graphII}, for each $k \in \{1,\dots,e-1\}$, we have
    \[  \mathsf c_k(\operatorname{Char}_{A(\Gamma^\prime)})\equiv
            \begin{cases}
             \mathsf c_k(\operatorname{Char}_{A(\Gamma)})+d_k2^{e-1-k}, &\text{if } k\leqslant e-2; \\
             \mathsf c_k(\operatorname{Char}_{A(\Gamma)}) , & \text{if } k=e-1
            \end{cases} \mod{2^{e-k}}.
        \]
    Furthermore, for each $k \in \{1,\dots,e-2\}$, we have
        \[  \mathbf 1^\transpose {A(\Gamma^\prime)^k}\mathbf 1\equiv
            \begin{cases}
             \mathbf 1^\transpose {A(\Gamma)^k} \mathbf 1, & \text{if } k\leqslant e-3;\\
             \mathbf 1^\transpose {A(\Gamma)^k} \mathbf 1+2d_{e-1}, & \text{if } k=e-2
            \end{cases} \mod{2^{e-k}}.
        \]

       Hence, using Lemma~\ref{lem:coefficient_link_formula} and Proposition~\ref{pro:evensumwalks}, when $k \in \{1,\dots,e-2\}$, we have
        \begin{align*}
                \mathsf c_k(\operatorname{Char}_{J-2A(\Gamma^\prime)})&= (-2)^k(\mathsf c_i(\operatorname{Char}_{A(\Gamma^\prime)})+\frac{1}{2}\sum_{i=1}^{k}\mathsf c_{k-i}(\operatorname{Char}_{A(\Gamma^\prime)}) \mathbf 1^\transpose A(\Gamma^\prime)^{i-1} \mathbf 1)\\
                &\equiv (-2)^k(\mathsf c_k(\operatorname{Char}_{A(\Gamma)})+d_k2^{e-1-k})-\sum_{i=1}^{k}\mathsf c_{k-i}(\operatorname{Char}_{A(\Gamma^\prime)})(-2)^{k-1}\cdot \mathbf 1^\transpose A(\Gamma)^{i-1}\mathbf 1 \mod{2^e}\\
                &\equiv (-2)^k\mathsf c_{k}(\operatorname{Char}_{A(\Gamma)})+\sum_{i=1}^{k}\mathsf c_{k-i}(\operatorname{Char}_{A(\Gamma)})(-2)^{k-1}\cdot \mathbf 1^\transpose A(\Gamma)^{i-1}\mathbf 1 +(-1)^kd_k2^{e-1}\mod{2^e}\\
                &\equiv \mathsf c_{k}(\operatorname{Char}_{J-2A(\Gamma)})+d_k2^{e-1} \equiv c_k \mod{2^e}
            \end{align*}
        Similarly, 
                \begin{align*}
                \mathsf c_{e-1}(\operatorname{Char}_{J-2A(\Gamma^\prime)})&= (-2)^{e-1}(\mathsf c_{e-1}(\operatorname{Char}_{A(\Gamma^\prime)})+\frac{1}{2}\sum_{i=1}^{e-1}\mathsf c_{e-1-i}(\operatorname{Char}_{A(\Gamma^\prime)}) \mathbf 1^\transpose A(\Gamma^\prime)^{i-1} \mathbf 1)\\
                &\equiv (-2)^{e-1}\mathsf c_{e-1}(\operatorname{Char}_{A(\Gamma)})-\sum_{i=1}^{e-1}\mathsf c_{e-1-i}(\operatorname{Char}_{A(\Gamma^\prime)})(-2)^{e-2}\cdot \mathbf 1^\transpose A(\Gamma^\prime)^{i-1}\mathbf 1 \mod{2^e}\\
                &\equiv \mathsf c_{e-1}(\operatorname{Char}_{J-2A(\Gamma)})+d_{e-1}2^{e-1} \equiv c_{e-1} \mod{2^e}.
            \end{align*}

Therefore, for some $c_e$, we have $\left ( \overline {0}^{(e)},\overline {c_3}^{(e)},\overline {c_4}^{(e)},\dots,\overline {c_{e-1}}^{(e)},\overline {c_e}^{(e)} \right )\in \mathcal C_{e}(\mathsf S_{n^\prime})$ which is a subset of $\mathcal C_{e}(\mathsf S_N)$ by Proposition~\ref{pro:lifting_vertices} together with the fact that $N\equiv n^\prime \equiv n \mod{2^{e-2}}$. 
\end{proof}

Using Lemma~\ref{lem:2em3belowEven}, we can prove the even counterpart of Theorem~\ref{thm:m2odd}.

\begin{theorem}
\label{thm:m2even}
    Let $e \geqslant 3$ be an integer.
    Then there exists $N_e \in \mathbb N$ such that
$|\mathcal C_{e}(\mathsf S_n)| = 2^{\binom{e-2}{2}}$ for each even integer $n \geqslant N_e$. 
\end{theorem}
\begin{proof}
By Theorem~\ref{thm:symUB}, it suffices to show that $|\mathcal C_{e}(\mathsf S_n)| \geqslant 2^{\binom{e-2}{2}}$.
       We proceed by induction on $e$.
    For $e = 3$, the bound is trivial.
    For $e = 4$, we can take $N_4 = 4$.
    Then $\{([0]_4,[0]_4,[0]_4),([0]_4,[8]_4,[0]_4)\} \subseteq \mathcal C_4(\mathcal S_n)$ for each even $n \geqslant N_4$.
    Indeed, we have the graph $\Gamma_1 = nP_1$ and $\Gamma_2 = P_3 \cup (n-3)P_1$.
    For the former, one finds that $\mathsf c_3(\operatorname{Char}_{J-2A(\Gamma_1)}) \equiv 0 \mod {16}$ and for the latter, $\mathsf c_3(\operatorname{Char}_{J-2A(\Gamma_2)}) \equiv 8 \mod {16}$.

    Now we can assume $e \geqslant 5$ for the inductive step.
    Take $N_{e}=N_{e-1}+\sum_{f=3}^{e-2}(e-f-1)n_f+m$ where $n_f$ is the order of an $(e,f)$-lift graph of type I and $m$ is the order of an $(e-2)$-lift graph of type II.

    Suppose $n = N_e + 2k$ for some nonnegative integer $k$.
    Then, by the implicit inductive hypothesis, we have $|\mathcal C_{e-1}(\mathsf S_{N_{e-1}+2k})| \geqslant 2^{\binom{e-3}{2}}$.
    By Lemma~\ref{lem:2em3belowEven}, we have $|\mathcal C_e(\mathsf S_n)| \geqslant 2^{e-3}2^{\binom{e-3}{2}} = 2^{\binom{e-2}{2}}$.
\end{proof}

When $n$ is even, Lemma~\ref{lem:first3coeffs} can be extended to the statement that says $2^k$ divides $\mathsf c_k(\operatorname{Char}_M)$ for each $M \in \mathsf S_n$.
Combining with Theorem~\ref{thm:m2even} reveals that, for even $n$ large enough, we must have
        \[
    \mathcal C_{e}(\mathsf S_n) = \left \{ \left (\overline {0}^{(e)},\overline{a_3}^{(e)}, \overline{a_4}^{(e)},\dots,\overline{a_{e-1}}^{(e)}, \overline {0}^{(e)} \right ) \; : \; 2^i \text{ divides } a_i \text{ for each $i \in \{3,\dots,e-1\}$} \right \}.
    \]
When $n$ is odd, however, the analogous expression is more complicated and relies on a modular relation between the coefficients (see \cite[Lemma 3.12]{GY19}).

\section{Lift tournaments}
\label{sec:lt}

In this section, we define two types of lift tournaments and provide constructions of families of these tournaments.
First, we give the main definitions.

\subsection{Definitions}

For $e, f \in \mathbb N$, we call a tournament $\Gamma$ an $(e,f)$-\textbf{lift tournament of type I} if 
\begin{enumerate}
    \item[(LT1a)] $2^{e-k}$ divides $\mathbf 1^\transpose A(\Gamma)^k \mathbf 1$ for each $k \in \{0,\dots,e-1\}$;
    \item[(LT1b)] for each $k \in \{1,\dots,e-1\}$,
\[
\mathsf c_k(\operatorname{Char}_{A(\Gamma)}) \equiv \begin{cases}
        0,  & \text{ if } k \ne f;\\
        2^{e-f-1},  & \text{ if $k = f$}
    \end{cases} \mod{2^{e-k}}.
\]
\end{enumerate}
Our constructions of lift tournaments of type I are obtained from Lemma~\ref{lem:Tt1ffp1} and Lemma~\ref{lem:inductiveTLiftConst}.

For an integer $e \geqslant 2$, we call a tournament $\Gamma$ an $e$-\textbf{lift tournament of type II} if 
\begin{enumerate}
    \item[(LT2a)] $2^{e}$ divides $\mathbf 1^\transpose A(\Gamma)^k \mathbf 1$ for each $k \in \{0,\dots,e-3\}$;
    \item[(LT2b)] $\mathbf 1^\transpose A(\Gamma)^{e-2} \mathbf 1 \equiv 2 \mod 4$ and $\mathbf 1^\transpose A(\Gamma)^{e-1} \mathbf 1 \equiv 1 \mod 2$;
    \item[(LT2c)] $2^{e}$ divides $\mathsf c_{k}(\operatorname{Char}_{A(\Gamma)})$ for each $k \in \{1,\dots,e-2\}$.
\end{enumerate}
Lemma~\ref{lem:Tt2} provides constructions of lift tournaments of type II.

\subsection{The join of two tournaments}

For tournaments $\Gamma_1$ and $\Gamma_2$, we will refer $\Gamma_1\oplus \Gamma_2$ as the tournament that has adjacency matrix 
\[
    A(\Gamma_1\oplus \Gamma_2) = \begin{pmatrix}
        A(\Gamma_1) & O \\
        J & A(\Gamma_2)
    \end{pmatrix}.
\]
This means that the tournament $\Gamma_1\oplus \Gamma_2$ has two components $\Gamma_1$ and $\Gamma_2$ and edges from each vertex of $\Gamma_2$ to $\Gamma_1$. 
For a positive integer $d$ and tournament $\Gamma$, we use $d\Gamma$ to denote $\bigoplus_{i=1}^d \Gamma$.

\begin{lemma}\label{lem:sumProdAdjTourn}
    Let $\Gamma_1$ and $\Gamma_2$ be tournaments.
    Then, for each $k\in \mathbb{N}$, we have 
    $$
    \mathbf{1}^\transpose A(\Gamma_1\oplus \Gamma_2)^{k}\mathbf{1} = \mathbf{1}^\transpose A(\Gamma_1)^{k}\mathbf{1}+\mathbf{1}^\transpose A(\Gamma_2)^{k}\mathbf{1}+\sum_{i=1}^{k}\mathbf{1}^\transpose A(\Gamma_1)^{i-1}\mathbf{1}\cdot \mathbf{1}^\transpose A(\Gamma_2)^{k-i}\mathbf{1}.
    $$
\end{lemma}
\begin{proof}
    The number of walks $\mathfrak w_k$ consisting of $k$ arcs is equal to $\mathbf{1}^\transpose A(\Gamma_1\oplus \Gamma_2)^{k}\mathbf{1}$.
    On the other hand, $\mathfrak w_k$ is the number of walks with $k$ arcs, each consisting of $k+1$ vertices. 
    Since there are no arcs from $\Gamma_1$ to $\Gamma_2$, each walk has the form $(v_1,\dots,v_r, w_1,\dots,w_{k+1-r})$ for some $r\in \{0,1,\dots,k+1\}$, where $w_j \in V(\Gamma_1)$ and $v_i \in V(\Gamma_2)$ for each $i \in \{1,\dots,r\}$ and $j \in \{1,\dots,k+1-r\}$.
    Therefore, we must have 
    \[
    \mathfrak w_k = \mathbf{1}^\transpose A(\Gamma_1)^{k}\mathbf{1}+\mathbf{1}^\transpose A(\Gamma_2)^{r}\mathbf{1}+\sum_{r=1}^{k}\mathbf{1}^\transpose A(\Gamma_1)^{r-1}\mathbf{1}\cdot \mathbf{1}^\transpose A(\Gamma_2)^{k-r}\mathbf{1},
    \]
    as required.
\end{proof}

\subsection{The walk polynomial}
\label{sec:wp}

    Given a tournament $\Gamma$ and a positive integer $d$, we define the \textbf{walk polynomial} $\operatorname{Walk}_{\Gamma,d}(x)$ by
\[
\operatorname{Walk}_{\Gamma,d}(x):=x^d+\sum^{d}_{i=1}\mathbf{1}^\transpose A(\Gamma)^{i-1}\mathbf{1}x^{d-i}.
\]
Note that the coefficient of $x^{d-k}$ in $\operatorname{Walk}_{\Gamma,d}(x)$ is equal to the number of walks in $\Gamma$ on $k$ vertices, that is, $\mathsf c_k(\operatorname{Walk}_{\Gamma,d}) = \mathbf{1}^\transpose A(\Gamma)^{k-1}\mathbf{1}$ for $k \in \{1,\dots,d\}$ and $\mathsf c_0(\operatorname{Walk}_{\Gamma,d}) = 1$.

\begin{corollary}
\label{cor:powerSumSum}
    Let $d$ be a positive integer, and let $\Gamma_1$ and $\Gamma_2$ be tournaments.
    Suppose that $\Gamma=\Gamma_1\oplus \Gamma_2$.
    Then $\operatorname{Walk}_{\Gamma,d}(x) = \operatorname{Walk}_{\Gamma_1,d}(x)\operatorname{Walk}_{\Gamma_2,d}(x)$.
    Furthermore, for each $k\in \{1,\dots,d\}$, we have 
    $$\mathsf p_k(\operatorname{Walk}_{\Gamma,d})=\mathsf p_k(\operatorname{Walk}_{\Gamma_1,d})+\mathsf p_k(\operatorname{Walk}_{\Gamma_2,d}).
    $$
\end{corollary}
\begin{proof}
    By Lemma~\ref{lem:sumProdAdjTourn}, we have $\operatorname{Walk}_{\Gamma,d}(x) = \operatorname{Walk}_{\Gamma_1,d}(x)\operatorname{Walk}_{\Gamma_2,d}(x)$, from which the lemma immediately follows.
\end{proof}

\subsection{Almost transitive tournaments}

We define the tournament $T_t$ as the tournament having $t+2$ vertices and, for each $i$ and $j$ satisfying $0\leqslant i<j \leqslant t+1$,
    \[  A(T_t)[i,j]=
            \begin{cases}
             0, & \text{if } i=0 \text{ and } j=t+1,\\
             1, & \text{otherwise};
            \end{cases} \quad \text{ and } \quad  A(T_t)[j,i]=
            \begin{cases}
             1, & \text{ if } i=0 \text{ and } j=t+1,\\
             0, & \text{ otherwise}.
            \end{cases}
        \]
The tournament $T_t$ can also be obtained by taking a transitive tournament with vertex set $\{0,1,\dots,t+1\}$ and replacing the arc $(0,t+1)$ with the arc $(t+1,0)$.
See Figure~\ref{fig:T3} for a picture of $T_3$, where an arrow from vertex $v$ to vertex $w$ indicates the arc $(v,w)$.
Accordingly, we may refer to the tournaments $T_t$ as \textbf{almost transitive}.

        \begin{figure}[h]
    \centering
    \begin{tikzpicture}[>=latex, every node/.style={circle, draw, inner sep=1.2pt, minimum size=14pt}]
  \node (v1) at (90:2.5cm)  {$0$};
  \node (v2) at (18:2.5cm)  {$1$};
  \node (v3) at (-54:2.5cm) {$2$};
  \node (v4) at (-126:2.5cm){$3$};
  \node (v5) at (162:2.5cm) {$4$};

  \foreach \i/\j in {
    1/2,1/3,1/4,5/1,
    2/3,2/4,2/5,
    3/4,3/5,
    4/5}
    \draw[->, line width=0.8pt, shorten >=1pt, shorten <=1pt] (v\i) -- (v\j);
\end{tikzpicture}
    \caption{The tournament $T_3$.}
    \label{fig:T3}
\end{figure}
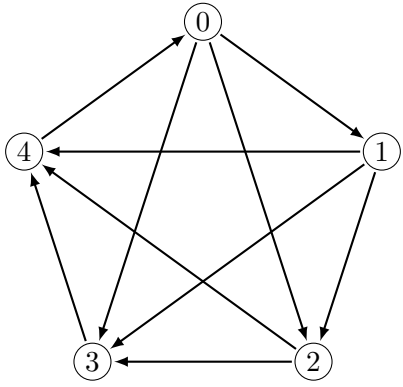

We use $\mathfrak S(n)$ to denote the symmetric group on $n$ symbols and $\operatorname{sgn}(\sigma)$ to denote the signature of a permutation $\sigma \in \mathfrak S(n)$.
Next, we give a formula for the characteristic polynomial of $A(T_t)$.

\begin{lemma}
\label{lem:characteristic_polynomial_of_Gt}
    Let $t\in \mathbb{N}$.
    Then
    \[
        \operatorname{Char}_{A(T_t)}(x)=x^{t+2}-\sum_{i=1}^{t}\binom{t}{i} \cdot x^{t-i}.
    \]
\end{lemma}
\begin{proof}
    Let $M=xI-A(T_t)$.
    Then Leibniz's formula yields
    $$\operatorname{Char}_{A(T_t)}(x)=\sum_{\sigma \in \mathfrak S(t+2)}\operatorname{sgn}(\sigma)\prod_{i=0}^{t+1}[M]_{i,\sigma(i)}.$$
    We examine all permutations $\sigma \in \mathfrak S(t+2)$ for which $\prod_{i=0}^{t+1}[M]_{i,\sigma(i)}$ is not $0$.
    Observe that $\sigma(t+1) \in \{0,t+1\}$.

    If $\sigma(t+1)=t+1$, then $\sigma(t)=t$ since only $[M]_{t,t}$ and $[M]_{t,t+1}$ are the only non-zero entries on the $t$th row of $M$.
    Similarly, we must have $\sigma(t-1)=t-1$ since only $[M]_{t-1,t-1}$, $[M]_{t-1,t}$, and $[M]_{t-1,t+1}$ are non-zero. 
    We can repeat this idea to deduce that $\sigma$ must be the identity permutation, which corresponds to the term $x^{t+2}$.
    
     If $\sigma(t+1)=0$, let $\gamma = (t+1,0,a_1,\dots,a_s)$ be the cycle of $\sigma$ that contains $t+1$.
     Since $[M]_{0,t+1} = 0$, we must have $s \geqslant 1$.
     We claim that $\sigma(i)=i$ for each $i$ outside of $\gamma$.
     Indeed, for all $j\neq t+1$ and $i<j$, we have $[M]_{j,i}=0$.
     Furthermore, it also follows that $a_i < a_{i+1}$ for each $i \in \{1,\dots,s-1\}$.
     The permutation $\sigma$ corresponds to the monomial $-x^{t-s}$, since $\operatorname{sgn}(\sigma)=(-1)^{s-1}$ and $[M]_{t+1,0}=[M]_{0,a_1}=[M]_{a_1,a_2}=\dots =[M]_{a_{s-1},a_s}=[M]_{a_s,t+1}=-1$. 
      For a fixed $t$, there are $\binom{t}{s}$ ways to choose $a_1,a_2,\dots,a_s$, which completes the proof.
\end{proof}

Observe that, by Lemma~\ref{lem:characteristic_polynomial_of_Gt}, the coefficients of $\operatorname{Char}_{A(T_t)}(x)$ are polynomials in $t$.
For the rest of this section, we will consider $t$ to be a variable.
To avoid confusion, we augment our notation as follows.
Denote by $\deg_t(p)$ the degree of $p(t)$, where $p(t)$ is a polynomial in the variable $t$.
Similarly, we denote by $\mathsf c_k^{(t)}(p)$ the coefficient of $t^{\deg_t(p)-k}$ in $p(t)$.
When the variable is clear from context, to avoid clutter, we continue to use the plain notation as used above.

\begin{corollary}
\label{cor:pChar}
    Let $k \geqslant 2$ be an integer.
    Then $\mathsf p_k(\operatorname{Char}_{A(T_t)})$ is a polynomial in $t$ of degree at most $k-2$.
    Furthermore, for $k \geqslant 3$, the degree $\deg_t(\mathsf p_k(\operatorname{Char}_{A(T_t)}))= k-2$ with top coefficient $\mathsf c_0^{(t)}(\mathsf p_k(\operatorname{Char}_{A(T_t)})) = \frac{k}{(k-2)!}$.
\end{corollary}
\begin{proof}
    By Lemma~\ref{lem:characteristic_polynomial_of_Gt}, $\mathsf c_1(\operatorname{Char}_{A(T_t)})=\mathsf c_2(\operatorname{Char}_{A(T_t)})=0$ and, for $k \geqslant 2$, we have $\mathsf c_k(\operatorname{Char}_{A(T_t)})=-{\binom{t}{k-2}}=-\frac{t(t-1)\dots(t-k+3)}{(k-2)!}$ which is a polynomial of degree $k-2$ for variable $t$. 
    Inductively applying \eqref{eqn:power_sum_elementary_symmetric_sum_relation} yields the conclusion of the lemma.
\end{proof}

Next, we show that the coefficients of the walk polynomial $T_t$ can also be expressed as polynomials in $t$.

\begin{lemma}
\label{lem:eWalk}
Let $d\geqslant 2$ be an integer.
Then, for each $k \in \{1,\dots,d\}$,
the coefficient $\mathsf c_k(\operatorname{Walk}_{T_t,d})$ is a polynomial in $t$ with
\begin{itemize}
    \item $\mathsf c_i(\operatorname{Walk}_{T_t,d})=\binom{t+2}{i}$ for each $i \in \{1,2\}$;
    \item $\deg(\mathsf c_k(\operatorname{Walk}_{T_t,d})) = k$;
    \item $\mathsf c^{(t)}_i \left (\mathsf c_k(\operatorname{Walk}_{T_t,d}) \right ) = \mathsf c_i^{(t)} \left ({\binom{t}{k}}+2{\binom{t}{k-1}}+{\binom{t}{k-2}}+ \frac{2^{k-2}}{(k-2)!}t^{k-2} \right )$ for $k \geqslant 3$ and $i \in \{0,1,2\}$.
\end{itemize}
\end{lemma}
\begin{proof}
    It is straightforward to calculate $\mathsf c_1(\operatorname{Walk}_{T_t,d})=t+2$ and $\mathsf c_2(\operatorname{Walk}_{T_t,d})=\binom{t+2}{2}$.
    Observe that $\mathsf c_k(\operatorname{Walk}_{T_t,d})=\mathbf 1^\transpose A(T_t)^{k-1}\mathbf 1$ is equal to the number of walks in $T_t$ consisting of $k$ vertices.
    For a walk with $k$ vertices $ \mathbf x = (x_1,x_2,\dots,x_k)$ we first consider relative locations for the vertices $0$ and $t+1$ in $\mathbf x$.

    If $x_j = 0$ with $j > 1$ then we must have $x_{j-1} = t+1$ and, if $j < k$ then $x_{j+1} \ne t+1$.
    Similarly, if $x_j = t+1$ with $j < k$ then we must have $x_{j+1} = 0$ and, if $j > 1$ then $x_{j-1} \ne 0$.
    There are $\binom{t}{r}$ possibilities for $x_{k+1}, x_{k+2}, \dots, x_{k+r}$ where each vertex is not in $\{0,t+1\}$.

    Now, we partition the walk $\mathbf x$ by the positions of $0$ and $t+1$.
    Let $m_1, m_2,\dots,m_s$ be the numbers of consecutive entries $x_j \not \in \{0,t+1\}$ of $\mathbf x$ between entries in $\{0,t+1\}$.
    For example, if $k = 10$ and $x_2=t+1$, $x_3=0$, $x_7=t+1$, $x_8 = 0$, and $x_{10}=t+1$ then $\mathbf x$ is partitioned into three (nonempty) parts with $m_1 = 1$, $m_2=3$, and $m_3=1$.
    Given $m_1, m_2,\dots,m_s$, the number of corresponding walks is equal to ${\binom{t}{m_1}}{\binom{t}{m_2}}\cdots{\binom{t}{m_s}}$, which is a polynomial in $t$ of degree $m_1 + m_2 + \dots + m_s \leqslant k$.
    When $s = 1$ (there are no occurrences of $0$ and $t+1$ in $\mathbf x$), then $m_s = k$ and the number of corresponding walks is $\binom{t}{k}$, which is a polynomial in $t$ of degree $k$.
    Hence $\mathsf c_k(\operatorname{Walk}_{T_t,d})$ is a polynomial in $t$ of degree $k$.

    Finally, we focus on the top three coefficients of $\mathsf c_r(\operatorname{Walk}_{T_t,d})$, which depends on the number of walks with $m_1 + m_2 + \dots + m_s \geqslant k-2$.
    This is equivalent to considering walks with at most two occurrences of $0$ and $t+1$.
    The number of walks with zero occurrences of $0$ and $t+1$ is $\binom{t}{k}$.
    The walks with just one occurrence of $0$ and $t+1$ have either $x_1 = 0$ or $x_k = t+1$.
    In each case, the number of corresponding walks is $\binom{t}{k-1}$.
    The walks with two occurrences of $0$ and $t+1$ have the form $x_1 = 0$ and $x_k=t+1$ or the form $x_i = t+1$ and $x_{i+1}=0$ for each $i \in \{1,\dots,k-1\}$.
    The number of corresponding walks is $\binom{t}  {k-2}$ or $\binom{t} {i-1} \binom{t}{k-1-i}$ respectively.

    The lemma follows since
    \[
    {\binom{t}{k}}+2{\binom{t}{k-1}}+{\binom{t}{k-2}}+\sum_{i=1}^{k-1}{\binom{t}{i-1}}{\binom{t}{k-1-i}} = {\binom{t}{k}}+2{\binom{t}{k-1}}+{\binom{t}{k-2}}+ \frac{2^{k-2}}{(k-2)!}t^{k-2}. \qedhere
    \]
\end{proof}

This subsection culminates in a technical result on expressing the power sums $\mathsf p_k(\operatorname{Walk}_{T_t,d})$ as polynomials in $t$.

\begin{lemma}
\label{lem:pWalk}
 Let $d\in\mathbb{N}$ with $d \geqslant 3$.
 For each $k \in \mathbb N$, the power sum $\mathsf p_k(\operatorname{Walk}_{T_t,d})$ is a polynomial in $t$ with
 \begin{itemize}
     \item $\mathsf p_1(\operatorname{Walk}_{T_t,d})=-t-2$, $\mathsf p_2(\operatorname{Walk}_{T_t,d})=t+2$, and $\mathsf p_3(\operatorname{Walk}_{T_t,d})=-7t-2$;
     \item $\deg_t(\mathsf p_k(\operatorname{Walk}_{T_t,d}))\leqslant k-3$ if $k \geqslant 4$ and $k$ is even;
     \item $\deg_t(\mathsf p_k(\operatorname{Walk}_{T_t,d}))=k-2$ if $k \geqslant 3$ and $k$ is odd;
     \item $\mathsf c_{0}^{(t)} (\mathsf p_k(\operatorname{Walk}_{T_t,d}))=\frac{-2k}{(k-2)!}$ if $k \geqslant 4$ and $k$ is odd.
 \end{itemize}
\end{lemma}
\begin{proof}
    By Lemma~\ref{lem:eWalk}, $\mathsf c_1(\operatorname{Walk}_{T_t,d})=t+2$, $\mathsf c_2(\operatorname{Walk}_{T_t,d})=\binom{t+2}{2}$, and $\mathsf c_3(\operatorname{Walk}_{T_t,d})={\binom{t}{3}}+2{\binom{t}{2}}+3{\binom{t}{1}}$. 
    Using \eqref{eqn:power_sum_elementary_symmetric_sum_relation}, we calculate $\mathsf p_1(\operatorname{Walk}_{T_t,d})=-t-2$, $\mathsf p_2(\operatorname{Walk}_{T_t,d})=t+2$, and $\mathsf p_3(\operatorname{Walk}_{T_t,d})=-7t-2$. 
    Furthermore, Lemma~\ref{lem:eWalk} and \eqref{eqn:power_sum_elementary_symmetric_sum_relation} together imply that $\mathsf p_k(\operatorname{Walk}_{T_t,d})$ is a polynomial in $t$.
    
    Now, we prove by induction that for $k\geqslant 4$,
if $k$ is even then $\deg_t(\mathsf p_k(\operatorname{Walk}_{T_t,d}))<k-2$ otherwise, if $k$ is odd then $\deg_t(\mathsf p_k(\operatorname{Walk}_{T_t,d}))=k-2$ with $\mathsf c_{0}^{(t)}(\mathsf p_k(\operatorname{Walk}_{T_t,d}))=\frac{2k}{(k-2)!}$.
     Assume the lemma is true for each $k=j$ such that $4\leqslant j<k^\prime$ for some $k^\prime \in \mathbb N$. 
     Now suppose $k = k^\prime$.
     Firstly, we bound the degree $\deg_t(\mathsf p_k(\operatorname{Walk}_{T_t,d}))$. 
     From \eqref{eqn:power_sum_elementary_symmetric_sum_relation}, we have 
     \begin{equation}
         \label{eqn:new2}
         \mathsf p_k(\operatorname{Walk}_{T_t,d})=-\sum_{i=1}^{k-1}\mathsf c_i(\operatorname{Walk}_{T_t,d})\mathsf p_{k-i}(\operatorname{Walk}_{T_t,d})-k\mathsf c_k(\operatorname{Walk}_{T_t,d}).
     \end{equation}
    By the inductive hypothesis and Lemma~\ref{lem:eWalk}, the polynomials $\mathsf c_i(\operatorname{Walk}_{T_t,d})\mathsf p_{k-i}(\operatorname{Walk}_{T_t,d})$ and $k\mathsf c_k(\operatorname{Walk}_{T_t,d})$ both have degree at most $k$.
    Hence, $\mathsf p_{k}(\operatorname{Walk}_{T_t,d})$ has degree at most $k$.

    For a polynomial $p(t) \in \mathbb Q[t]$, let $\widehat {\mathsf c_r}(p)$ denote the coefficient of $t^r$ in $p(t)$.
    Next, we show that $\widehat{\mathsf c_k}(\mathsf p_{k}(\operatorname{Walk}_{T_t,d})) = \widehat{\mathsf c_{k-1}}(\mathsf p_{k}(\operatorname{Walk}_{T_t,d}))= 0$.
    
    \paragraph{Claim $\widehat{\mathsf c_k}(\mathsf p_{k}(\operatorname{Walk}_{T_t,d})) = 0$.}
    By the inductive hypothesis and Lemma~\ref{lem:eWalk}, for $i < k-1$, the polynomial $\mathsf c_i(\operatorname{Walk}_{T_t,d})\mathsf p_{k-i}(\operatorname{Walk}_{T_t,d})$ has degree at most $k-1$.
    By Lemma~\ref{lem:eWalk}, the polynomial $\mathsf c_{k-1}(\operatorname{Walk}_{T_t,d})\mathsf p_{1}(\operatorname{Walk}_{T_t,d})$ has degree $k$ with $\widehat{\mathsf c_k}(\mathsf c_{k-1}(\operatorname{Walk}_{T_t,d})\mathsf p_{1}(\operatorname{Walk}_{T_t,d})) = 1/(k-1)!$ and $\mathsf c_{k}(\operatorname{Walk}_{T_t,d})$ has degree $k$ with $\widehat{\mathsf c_k}(\mathsf c_{k}(\operatorname{Walk}_{T_t,d}) = 1/k!$.
    The claim now follows from \eqref{eqn:new2}.
     
     Define the falling factorial function $f_k(x) := \prod_{l=0}^{k-1}(x-l)$.
     Note that, for each $i \in \{1,\dots,k-1\}$,
     \begin{equation}
     \label{eqn:ef}
         \mathsf c_i(f_k) = \mathsf c_i(f_{k-1}) - (k-1)\mathsf c_{i-1}(f_{k-1}).
     \end{equation}

    \paragraph{Claim $\widehat{\mathsf c_{k-1}}(\mathsf p_{k}(\operatorname{Walk}_{T_t,d})) = 0$.}
    By the inductive hypothesis and Lemma~\ref{lem:eWalk}, for $i < k-2$, the polynomial $\mathsf c_i(\operatorname{Walk}_{T_t,d})\mathsf p_{k-i}(\operatorname{Walk}_{T_t,d})$ has degree at most $k-2$.
    Since $\mathsf p_{2}(\operatorname{Walk}_{T_t,d})) = t+2$, by Lemma~\ref{lem:eWalk}, we have $\widehat{\mathsf c_{k-1}}(\mathsf c_{k-2}(\operatorname{Walk}_{T_t,d})\mathsf p_{2}(\operatorname{Walk}_{T_t,d})) = 1/(k-2)!$.
    Since $\mathsf p_{1}(\operatorname{Walk}_{T_t,d})) = -t-2$, by Lemma~\ref{lem:eWalk}, we have
    \[
    \widehat{\mathsf c_{k-1}}(\mathsf c_{k-1}(\operatorname{Walk}_{T_t,d})\mathsf p_{1}(\operatorname{Walk}_{T_t,d})) = -\mathsf c_1(f_{k-1})/(k-1)!-2/(k-1)!-2/(k-2)!.
    \]
    Again, by Lemma~\ref{lem:eWalk}, we have $\widehat{\mathsf c_{k-1}}(\mathsf c_{k}(\operatorname{Walk}_{T_t,d})) = \mathsf c_1(f_k)/k!+2/(k-1)!$.
    The claim now follows from \eqref{eqn:new2} and \eqref{eqn:ef}.

    It remains to prove our final claim: 
    \[
    \widehat{\mathsf c_{k-2}}(\mathsf p_{k}(\operatorname{Walk}_{T_t,d})) = \begin{cases}
        0 & \text{ if $k$ is even;} \\
        \frac{-2k}{(k-2)!} & \text{ if $k$ is odd.}
    \end{cases}
    \]
    By the inductive hypothesis and Lemma~\ref{lem:eWalk}, for $i < k-3$ such that $k-i$ is even, it follows that the polynomial $\mathsf c_i(\operatorname{Walk}_{T_t,d})\mathsf p_{k-i}(\operatorname{Walk}_{T_t,d})$ has degree at most $k-3$.
    Similarly, for $i < k-3$ such that $k-i$ is odd, we have 
    \[
    \widehat{\mathsf c_{k-2}}(\mathsf c_{i}(\operatorname{Walk}_{T_t,d})\mathsf p_{k-i}(\operatorname{Walk}_{T_t,d})) = \frac{-2(k-i)}{(k-i-2)!i!}.
    \]
    Next, we find that
    \[
    \sum_{\substack{1 \leqslant i \leqslant k-4 \\ k-i \text{ odd }}}\frac{2(k-i)}{(k-i-2)!i!}  = \frac{2^{k-3}}{(k-3)!}+\frac{2^{k-1}}{(k-2)!}-\frac{6}{(k-3)!}- \begin{cases}
        0, & \text{ if $k$ is even;} \\
        \frac{2k}{(k-2)!}, & \text{ if $k$ is odd,} \\
    \end{cases}
    \]
    which follows from a straightforward induction, using the observation that
    \begin{align*}
       \sum_{\substack{1 \leqslant i \leqslant k-4 \\ k-i \text{ odd }}}\frac{2(k-i)}{(k-i-2)!i!}  &= \frac{2}{(k-3)!} \sum_{\substack{1 \leqslant i \leqslant k-4 \\ k-i \text{ odd }}}\binom{k-3}{i}+ \frac{4}{(k-2)!} \sum_{\substack{1 \leqslant i \leqslant k-4 \\ k-i \text{ odd }}}\binom{k-2}{i}.
    \end{align*}
    Furthermore,
    \begin{align*}
        \widehat{\mathsf c_{k-2}}(\mathsf c_{k-3}(\operatorname{Walk}_{T_t,d})\mathsf p_{3}(\operatorname{Walk}_{T_t,d})) &= \frac{-7}{(k-3)!}; \\
        \widehat{\mathsf c_{k-2}}(\mathsf c_{k-2}(\operatorname{Walk}_{T_t,d})\mathsf p_{2}(\operatorname{Walk}_{T_t,d})) &= \frac{\mathsf c_1(f_{k-2})+2}{(k-2)!}+\frac{2}{(k-3)!}; \\
        \widehat{\mathsf c_{k-2}}(\mathsf c_{k-1}(\operatorname{Walk}_{T_t,d})\mathsf p_{1}(\operatorname{Walk}_{T_t,d})) &= -\frac{\mathsf c_2(f_{k-1})+2\mathsf c_1(f_{k-1})}{(k-1)!}-\frac{4+2\mathsf c_1(f_{k-2})}{(k-2)!}-\frac{2^{k-3}+1}{(k-3)!}; \\
        \widehat{\mathsf c_{k-2}}(\mathsf c_{k}(\operatorname{Walk}_{T_t,d}) &= \frac{\mathsf c_2(f_{k})}{k!}+\frac{2\mathsf c_1(f_{k-1})}{(k-1)!}+\frac{2^{k-2}+1}{(k-2)!}.
    \end{align*}
    Using \eqref{eqn:ef}, we have
    \[
    \widehat{\mathsf c_{k-2}} \left ( k\mathsf c_{k}(\operatorname{Walk}_{T_t,d}) + \sum_{i=1}^3\mathsf c_{k-i}(\operatorname{Walk}_{T_t,d})\mathsf p_{i}(\operatorname{Walk}_{T_t,d}) \right ) = \frac{2^{k-3}}{(k-3)!}+\frac{2^{k-1}}{(k-2)!}-\frac{6}{(k-3)!}.
    \]
    The claim then follows from \eqref{eqn:new2}.
    \end{proof}

\subsection{Construction of lift tournaments}

Now we are ready to provide constructions of lift tournaments of types I and II.

Given a polynomial $p(x) \in \mathbb Z[x]$, denote by $\nu_{\min}(p)$ the minimum of the $2$-valuations of its coefficients, i.e., 
\[
\nu_{\min}(p) := \min \{ \nu_2(\mathsf c_i(p)) \; : \; i \in \{0,\dots, \deg p\}\}.
\]

The proof of the next result is routine and left to the reader.

\begin{proposition}
\label{pro:polySum}
Let $f, m \in \mathbb N$ with $e \geqslant 2$ and $(a_0,\dots,a_{e-2}) \in \mathbb N^{e-1}$ such that, for each $t \in \{0,\dots,e-2\}$,
\[
a_0\cdot0^t+a_1\cdot1^t+\cdots+a_{e-2}\cdot (e-2)^t \equiv c_t \mod{2^m}
\]
for some $c_t \in \mathbb Z$.
Suppose $p(x) \in \mathbb Q[x]$ with degree at most $e-2$ and $\nu_{\min}(p) \geqslant -m$.
Then
\[
a_0\cdot p(0)+a_1\cdot p(1)+\cdots+a_{e-2}\cdot p(e-2)\equiv \mathsf c_0(p)\cdot c_{0}+\mathsf c_{1}(p)\cdot c_{1}+\dots+\mathsf c_{e-2}(p)\cdot c_{e-2} \mod{2^{m+\nu_{\min}(p)}}.
\]
\end{proposition}

Lemma~\ref{lem:pWalk} and Corollary~\ref{cor:pChar} allow us to establish the following existence result about tournaments whose characteristic polynomial and walk polynomial have certain properties.
These properties will be used to show the existence of lift tournaments of type I and type II.

\begin{proposition}
\label{pro:tournyP}
    Let $e \geqslant 4$ be an integer.
    Then there exists $(a_0,\dots,a_{e-2}) \in \mathbb N^{e-1}$ such that, for each $k \in \{1,\dots,e\}$, the tournament $\Gamma = \bigoplus_{t=0}^{e-2} a_t T_t$ satisfies
    \begin{align*}
        \mathsf p_k(\operatorname{Walk}_{\Gamma,e}) &\equiv \begin{cases}
            0, & \text{ $k \in \{1,\dots,e-1\}$} \\
            0, & \text{ $k = e$ and $e$ is even} \\
            -2e, & \text{ $k = e$ and $e$ is odd}
        \end{cases} \mod{2^{e+1+\nu_2((e+2)!)}}; \\
        \mathsf p_k(\operatorname{Char}_{A(\Gamma)}) &\equiv \begin{cases}
            0, & \text{ $k \in \{1,\dots,e-1\}$} \\
            e, & \text{ $k = e$}
        \end{cases} \mod{2^{e+1+\nu_2((e+2)!)}}.
    \end{align*}
\end{proposition}
\begin{proof}
        Suppose that $$
    m = e+1+\nu_2((e+2)!) - \min \left \{ \min \{\nu_{\min}(\mathsf p_k(\operatorname{Char}_{A(T_t)})),\nu_{\min}(\mathsf p_k(\operatorname{Walk}_{T_t,e}))\} \; : \; k \in \{0,\dots,e\}\right \}.
    $$
    It is straightforward to verify~\cite[Section 1.2.3, Exercise 40]{knuth1997art} that
    \[
    \sum_{i=0}^{e-2} (-1)^{e-2-i} \binom{e-2}{i}i^t = \begin{cases}
        0, & \text{ if $t \in \{0,\dots,e-3\}$}; \\
        (e-2)!, & \text{ if $t = e-2$.}
    \end{cases}
    \]
    Thus, there exists $(a_0,\dots,a_{e-2}) \in \mathbb N^{e-1}$ such that 
    \begin{equation}
        \label{eqn:sumCongr}
        a_0\cdot0^t+a_1\cdot1^t+\cdots+a_{e-2}\cdot (e-2)^t\equiv
            \begin{cases}
             0, & \text{ if $t \in \{0,\dots,e-3\}$}; \\
             (e-2)!, & \text{ if $t = e-2$}
            \end{cases} \mod{2^{m}}.
    \end{equation}
    Let $\Gamma=\bigoplus_{t=0}^{e-2} a_t T_t$.
    Then, by \eqref{eqn:powerProdSum} and  Corollary~\ref{cor:powerSumSum}, we have
    \[
    \mathsf p_k(\operatorname{Char}_{A(\Gamma)})=\sum_{t=0}^{e-2} a_t \mathsf p_k(\operatorname{Char}_{A(T_t)})
         \]
         and 
         \[
        \mathsf p_k(\operatorname{Walk}_{\Gamma,d})=\sum_{t=0}^{e-2} a_t \mathsf p_k(\operatorname{Walk}_{T_t,d}).
         \]
    For each $k \in \{0,1,\dots,e\}$, by Lemma~\ref{lem:pWalk} and Corollary~\ref{cor:pChar},  we have that both $\mathsf p_k(\operatorname{Walk}_{T_t,e})$ and $\mathsf p_k(\operatorname{Char}_{A(T_t)}) \in \mathbb Q[t]$, that is, we can express $\mathsf p_k(\operatorname{Walk}_{T_t,e})$ and $\mathsf p_k(\operatorname{Char}_{A(T_t)})$ as rational polynomials with variable $t$, but note that their coefficients may not be integers.
    
    Furthermore, by Lemma~\ref{lem:pWalk}, for each $k \in \{0,1,\dots,e\}$, we have $\deg_t \mathsf p_k(\operatorname{Walk}_{T_t,e}) \leqslant e-2$.  
    If $k$ is even then $\deg_t \mathsf p_k(\operatorname{Walk}_{T_t,e}) \leqslant k-3$, otherwise $\deg_t \mathsf p_k(\operatorname{Walk}_{T_t,e}) = k-2$ with $\mathsf c_{0}^{(t)}(\mathsf p_k(\operatorname{Walk}_{T_t,e})) = -2k/(k-2)!$.
    Similarly, by Corollary~\ref{cor:pChar}, for each $k \in \{0,1,\dots,e\}$, we have $\deg_t \mathsf p_k(\operatorname{Char}_{A(T_t)}) \leqslant k-2$ and $\deg_t \mathsf p_k(\operatorname{Char}_{A(T_t)}) = k-2$ with $\mathsf c_{0}^{(t)} \left (\mathsf p_k(\operatorname{Char}_{A(T_t)})\right ) = k/(k-2)!$.
    The lemma then follows from the application of Proposition~\ref{pro:polySum}.
\end{proof}

Now we can give our first construction of lift tournaments.

\begin{lemma}
\label{lem:Tt1ffp1}
    Let $e \geqslant 4$ be an even integer.
    Then there exists $(a_0,\dots,a_{e-2}) \in \mathbb N^{e-1}$ such that $\bigoplus_{t=0}^{e-2} a_t T_t$ is an $(e+1,e)$-lift tournament of type I.
\end{lemma}
\begin{proof}
    By Proposition~\ref{pro:tournyP}, there exists $(a_0,\dots,a_{e-2}) \in \mathbb N^{e-1}$ such that the tournament $\Gamma = \bigoplus_{t=0}^{e-2} a_t T_t$ satisfies
    \begin{align*}
        \mathsf p_k(\operatorname{Walk}_{\Gamma,e}) &\equiv 0 \mod{2^{e+1+\nu_2((e+2)!)}} \text{ for each } k \in \{1,\dots,e\}; \\
        \mathsf p_k(\operatorname{Char}_{A(\Gamma)}) &\equiv \begin{cases}
            0, & \text{ $k \in \{1,\dots,e-1\}$} \\
            (-1)^{e+1}e, & \text{ $k = e$}
        \end{cases} \mod{2^{e+1+\nu_2((e+2)!)}}.
    \end{align*}
    Now apply Lemma~\ref{lem:ptoe} to deduce $\mathsf c_k(\operatorname{Walk}_{\Gamma,e}) = \mathbf 1^\transpose A^{k-1} \mathbf 1 \equiv 0 \mod{2^{e+1}}$ for each $k \in \{1,\dots,e\}$.
    Whence, we have established (LT1a).
    
    Similarly, apply Lemma~\ref{lem:ptoe} to deduce that 
    $\mathsf c_k(\operatorname{Char}_{A(\Gamma)}) \equiv 0 \mod{2^{e+1}}$ for each $k \in \{1,\dots,e-1\}$ and $\mathsf c_e(\operatorname{Char}_{A(\Gamma)}) \equiv 1 \mod{2}$.
    Thus, $\Gamma$ satisfies (LT1b).
\end{proof}

Together with Lemma~\ref{lem:Tt1ffp1}, the following lemma allows us to construct $(e,f)$-lift tournaments of type I for all integers $e$ and $f$ satisfying $e \geqslant f+1 \geqslant 5$ with $f$ even. 

\begin{lemma}\label{lem:inductiveTLiftConst}
    If $\Gamma$ is an $(e,f)$-lift tournament of type I, then $2\Gamma$ is an $(e+1,f)$-lift tournament of type I.
\end{lemma}
\begin{proof}
    By Lemma~\ref{lem:sumProdAdjTourn},
        \[
            \mathbf{1}^\transpose A(2\Gamma)^{i}\mathbf{1} = 2\mathbf{1}^\transpose A(\Gamma)^{i}\mathbf{1}+\sum_{j=1}^{i}\mathbf{1}^\transpose A(\Gamma)^{j-1}\mathbf{1}\cdot \mathbf{1}^\transpose A(\Gamma)^{i-j}\mathbf{1}. 
        \]
        Since $2^{e-i+1}$ divides $2\mathbf{1}^\transpose A(\Gamma)^{i}\mathbf{1}$ and $2^{e-i+1} = 2^{e-(j-1)}\cdot 2^{e-(i-j)}$ divides $\mathbf{1}^\transpose A(\Gamma)^{j-1}\mathbf{1}\cdot \mathbf{1}^\transpose A(\Gamma)^{i-j}\mathbf{1}$, we find that $2\Gamma$ satisfies (LT1a).
    Since $\operatorname{Char}_{A(2\Gamma)}(x) = \operatorname{Char}_{A(\Gamma)}^2(x)$, we can deduce that $\mathsf c_k(\operatorname{Char}_{A(2\Gamma)}) \equiv 2 \mathsf c_k(\operatorname{Char}_{A(\Gamma)}) \mod {2^{e-k}}$ for each $k \in \{1,\dots,e-1\}$, from which (LT1b) follows.
\end{proof}

The last lemma of this subsection provides a construction of $e$-lift tournaments for each even integer $e \geqslant 6$.

\begin{lemma}
\label{lem:Tt2}
    Let $e \geqslant 6$ be an even integer.
    Then there exists $(a_0,\dots,a_{e-2}) \in \mathbb N^{e-1}$ such that $\bigoplus_{t=0}^{e-2} a_t T_t$ is an $e$-lift tournament of type II.
\end{lemma}
\begin{proof}
By Proposition~\ref{pro:tournyP}, there exists $(a_0,\dots,a_{e-2}) \in \mathbb N^{e-1}$ such that the tournament $\Gamma = \bigoplus_{t=0}^{e-2} a_t T_t$ satisfies
    \begin{align*}
        \mathsf p_k(\operatorname{Walk}_{\Gamma,e}) &\equiv \begin{cases}
            0, & \text{ $k \in \{1,\dots,e-2\}$} \\
            2(1-e), & \text{ $k = e-1$}
        \end{cases} \mod{2^{e+\nu_2((e+1)!)}}; \\
        \mathsf p_k(\operatorname{Char}_{A(\Gamma)}) &\equiv \begin{cases}
            0, & \text{ $k \in \{1,\dots,e-2\}$} \\
            e-1, & \text{ $k = e-1$}
        \end{cases} \mod{2^{e+\nu_2((e+1)!)}}.
    \end{align*}
    Using Lemma~\ref{lem:ptoe}, with $m = e+\nu_2((e+1)!)$ and $n = e-1$, we find that
    \begin{align*}
        \mathsf c_k(\operatorname{Walk}_{\Gamma,e}) = \mathbf 1^\transpose A(\Gamma)^{i-1} \mathbf 1 &\equiv \begin{cases}
            0 , & \text{ $k \in \{1,\dots,e-2\}$} \\
            -2, & \text{ $k = e-1$}
        \end{cases} \mod{2^{e+\nu_2((e+1)!)-\nu_2(k!)}}; \\
        \mathsf c_k(\operatorname{Char}_{A(\Gamma)}) &\equiv \begin{cases}
            0 , & \text{ $k \in \{1,\dots,e-2\}$} \\
            1, & \text{ $k = e-1$}
        \end{cases} \mod{2^{e+\nu_2((e+1)!)-\nu_2(k!)}}.
    \end{align*}
    This shows that $\Gamma$ satisfies (LT2a) and (LT2c).
    To establish (LT2b), it remains to show that $\mathbf 1^\transpose A(\Gamma)^{e-1} \mathbf 1$ is odd.

    Let $\Gamma^\prime = P_1 \oplus \Gamma$ and let $n^\prime$ be the order of $\Gamma^\prime$.
    On the other hand, by Lemma~\ref{lem:coefficient_link_formula},
    \begin{align}
    \label{eqn:GpWalki}
        \mathsf c_{e}(\operatorname{Char}_{J-2A(\Gamma^\prime)}) &= (-2)^{e}\left ( \mathsf c_{e}(\operatorname{Char}_{A(\Gamma^\prime)}) + \frac{1}{2}\sum_{i=1}^{e} \mathsf c_{e-i}(\operatorname{Char}_{A(\Gamma^\prime)})\mathbf 1^\transpose A(\Gamma^\prime)^{i-1} \mathbf 1 \right ).
    \end{align}
    Since $\operatorname{Char}_{\Gamma^\prime}(x) = x\operatorname{Char}_{\Gamma}(x)$, we have $\mathsf c_{k}(\operatorname{Char}_{A(\Gamma^\prime)}) = \mathsf c_{k}(\operatorname{Char}_{A(\Gamma)})$ for all $k \in \{ 0,1,\dots, n^\prime - 1 \}$.

    By Lemma~\ref{lem:sumProdAdjTourn}, we have
    \begin{align*}
        \mathbf{1}^\transpose A(\Gamma^\prime)^{i}\mathbf{1} &= \mathbf{1}^\transpose A(\Gamma)^{i}\mathbf{1}+\mathbf{1}^\transpose A(P_1)^{i}\mathbf{1}+\sum_{j=1}^{i}\mathbf{1}^\transpose A(\Gamma)^{j-1}\mathbf{1}\cdot \mathbf{1}^\transpose A(P_1)^{i-j}\mathbf{1}.
    \end{align*}
    It follows that $n^\prime = \mathsf p_1(\operatorname{Walk}_{\Gamma,e}) \equiv 1 \mod {2^{e+\nu_2((e+1)!)}}$, $\mathbf{1}^\transpose A(\Gamma^\prime)^{e-1}\mathbf{1} \equiv \mathbf{1}^\transpose A(\Gamma)^{e-1}\mathbf{1} \mod {2}$, and  $\mathbf{1}^\transpose A(\Gamma^\prime)^{i}\mathbf{1} \equiv 0 \mod 2$ for each $i \in \{1,\dots,e-2\}$.
    Therefore, \eqref{eqn:GpWalki} becomes
    \begin{align}
        \mathsf c_{e}(\operatorname{Char}_{J-2A(\Gamma^\prime)}) &\equiv 2^{e-1}(\mathsf c_{e-1}(\operatorname{Char}_{\Gamma^\prime})\mathbf{1}^\transpose A(\Gamma^\prime)^{0}\mathbf{1}+\mathsf c_{0}(\operatorname{Char}_{\Gamma^\prime})\mathbf{1}^\transpose A(\Gamma^\prime)^{e-1}\mathbf{1}) \mod{2^{e}} \nonumber \\
        &\equiv 2^{e-1}(1+\mathbf{1}^\transpose A(\Gamma^\prime)^{e-1}\mathbf{1}) \mod{2^{e}}. \label{eqn:in3}
    \end{align}
    Furthermore, by Lemma~\ref{lem:coefficient_link_formula}, we have 
    \begin{align*}
        \mathsf c_{k}(\operatorname{Char}_{J-2A(\Gamma^\prime)}) &\equiv \begin{cases}
            1 , & \text{ for } k \in \{0,1\} \\
            0 , & \text{ for } k \in \{2,\dots,e-1\}
        \end{cases} \mod{2^e}.
    \end{align*}
      Hence,
        \begin{equation}
            \label{eqn:in1}
            \sum^{e-1}_{i=0}{{n^\prime-i}\choose {n^\prime-e-1}}{\mathsf c_{i}(\operatorname{Char}_{J-2A(\Gamma^\prime)})}\equiv {{n^\prime }\choose {e+1}}+{{n^\prime-1}\choose {e}} \mod{2^{e}}.
        \end{equation}
        By Lemma~\ref{lem:TcoeffRel}, we have
        \begin{equation}
            \label{eqn:in2}
            \sum^{e-1}_{i=0}{{n^\prime-i}\choose {n^\prime-e-1}}{\mathsf c_{i}(\operatorname{Char}_{J-2A(\Gamma^\prime)})}+(n^\prime-e)\mathsf c_{e}(\operatorname{Char}_{J-2A(\Gamma^\prime)}) \equiv 0 \mod{2^{e}}.
        \end{equation}
    Next observe that $\nu_2({{n^\prime}\choose {e+1}}) \geqslant \nu_2(n^\prime-1)-\nu_2((e+1)!)\geqslant e$ and, similarly, $\nu_2({{n^\prime-1}\choose {e}})\geqslant \nu_2(n^\prime-1)-\nu_2((e+1)!)\geqslant e$. 
    Combining this with \eqref{eqn:in1} and \eqref{eqn:in2}, we find that $\mathsf c_{e}(\operatorname{Char}_{J-2A(\Gamma^\prime)}) \equiv 0 \mod{2^{e}}$.
    Finally, from \eqref{eqn:in3}, we deduce that $\mathbf{1}^\transpose A(\Gamma)^{e-1}\mathbf{1}$ is odd, as required.
\end{proof}

\section{Proof of Theorem~\ref{thm:m3}}
\label{sec:pm3}

In this section, we prove Theorem~\ref{thm:m3}.
First, we obtain the upper bound for the cardinality of $\mathcal C_e(\mathsf T_n)$.

\subsection{Upper bound}

\begin{lemma}
\label{lem:Tub}
Let $e \geqslant 4$ be an integer.
Then, for all $n \in \mathbb N$, we have
    \[
    |\mathcal C_e(\mathsf T_n)| \leqslant  \begin{cases}
        2^{\lfloor \frac{e-1}{2} \rfloor\lfloor \frac{e-2}{2} \rfloor}, & \text{ if $n$ is even}; \\
        2^{\lfloor \frac{e-2}{2} \rfloor\lfloor \frac{e-3}{2} \rfloor}, & \text{ if $n$ is odd}. \\
    \end{cases}
    \]
\end{lemma}
\begin{proof}
    Let $M \in \mathsf T_n$.
    First, we consider the case of $n$ even.
    We count the number of possible congruence classes of each coefficient $a_i$ modulo $2^e$.
    By Lemma~\ref{lem:first3coeffs}, the coefficients $\mathsf c_1(\operatorname{Char}_M)$ and $\mathsf c_2(\operatorname{Char}_M)$ are determined by $n$ and $2^{j-1}$ divides $\mathsf c_j(\operatorname{Char}_M)$ for each $j \in \{1,\dots,n\}$.
    For the $k$th coefficient $\mathsf c_k(\operatorname{Char}_M)$, suppose that the congruence classes for the $k$ previous coefficients $\mathsf c_0(\operatorname{Char}_M), \dots, \mathsf c_{k-1}(\operatorname{Char}_M)$ have been given.
    If $k$ is even, then, since $2^{k-1}$ divides $\mathsf c_k(\operatorname{Char}_M)$, there are $2^{e-k+1}$ possible congruence classes for $\mathsf c_k(\operatorname{Char}_M)$ modulo $2^e$.
    If $k$ is odd, then the congruence class for $\mathsf c_k(\operatorname{Char}_M)$ modulo $2^e$ is uniquely determined by \eqref{eqn:akeq}.
    Hence, 
    $$|\mathcal C_e(\mathsf U_n)| \leqslant 2^{(e-3)+(e-5)+\dots+(e-2\lfloor (e-1)/2\rfloor - 1)},$$
    as required.

    Finally, suppose that $n$ is odd.
    Using \eqref{eqn:akeq2}, we find that the congruence class of $\mathsf c_{k-1}(\operatorname{Char}_M)$ modulo $2^{k-1}$ is determined by the congruence classes of $\mathsf c_{0}(\operatorname{Char}_M),\dots,\mathsf c_{k-2}(\operatorname{Char}_M)$ modulo $2^{k-1}$.
    Now, in a similar way to the above, we obtain
    \[
    |\mathcal C_e(\mathsf U_n)| \leqslant 2^{(e-4)+(e-6)+\cdots +(e-2{\lfloor \frac{e-2}{2}\rfloor}-2)}. \qedhere
    \]
\end{proof}

\subsection{Lower bound}

We begin with an inclusion result on $\mathcal C_e(\mathsf T_n)$, which is analogous to Proposition~\ref{pro:lifting_vertices}.

\begin{proposition}\label{prop:lifting_vertices_tournament}
    Let $n, e\in \mathbb{N}$.
    Then $\mathcal C_e(\mathsf T_n) \subset \mathcal C_e(\mathsf T_{n+2^{e}})$.
\end{proposition}
\begin{proof}
    Let $ \mathbf a = \left ( \overline{a_2}^{(e)},\overline{a_3}^{(e)},\dots,\overline{a_e}^{(e)}\right) \in \mathcal C_e(\mathsf T_n)$ and 
    let $\Gamma$ be an $n$-vertex tournament such that the coefficient $\mathsf c_i(\operatorname{Char}_{J-2A(\Gamma)}) \equiv a_i \mod{2^e}$ for each $i \in \{2,\dots,e\}$.
    Let $\Gamma^\prime = \Gamma \oplus 2^{e}P_1$. 
Since $\operatorname{Char}_{A(\Gamma^\prime)}(x) = \operatorname{Char}_{A(\Gamma)}(x)x^{2^e}$, we have $\mathsf c_i(\operatorname{Char}_{A(\Gamma^\prime)})=\mathsf c_i(\operatorname{Char}_{A(\Gamma)})$ for each $i \in \{1,2,\dots,e\}$. 

By Lemma~\ref{lem:sumProdAdjTourn}, we have $\mathbf 1^\transpose A(\Gamma^\prime)^0 \mathbf 1 = \mathbf 1^\transpose A(\Gamma)^0 \mathbf 1 + 2^e$ and $\mathbf 1^\transpose A(\Gamma^\prime)^i \mathbf 1 = \mathbf 1^\transpose A(\Gamma)^i \mathbf 1 + 2^e\mathbf 1^\transpose A(\Gamma)^{i-1} \mathbf 1$ for $i \in \mathbb N$.
 By Lemma~\ref{lem:coefficient_link_formula} together with Lemma~\ref{lem:first3coeffs}, we have $\mathsf c_i(\operatorname{Char}_{J-2A(\Gamma^\prime)}) \equiv \mathsf c_i(\operatorname{Char}_{J-2A(\Gamma)}) \mod {2^e}$ for each $i \in \{1,2,\dots,e\}$.
 Hence, $ \mathbf a \in \mathcal C_e(\mathsf T_{n+2^e})$.
\end{proof}

The main utility of lift tournaments of type I is the following lemma.

\begin{lemma}\label{lem:tournament_lift_graph_effect}
Let $f\geqslant 4$ be an even integer and let $e > f$ be an integer.
Let $\Gamma$ be a tournament and $\Lambda$ be an $(e,f)$-lift tournament of type I.
Then, for each even $k \in \{2,3,\dots,e-1\}$, we have
    \[ \mathsf c_{k}(\operatorname{Char}_{J-2A(\Gamma \oplus \Lambda)}) \equiv
    \begin{cases}
        \mathsf c_{k}(\operatorname{Char}_{J-2A(\Gamma)}) , & \text{ if } k\ne f\\
        \mathsf c_{k}(\operatorname{Char}_{J-2A(\Gamma)})+2^{e-1}, & \text{ if $k = f$}
    \end{cases} \mod{2^{e}}.
    \]
\end{lemma}

\begin{proof}   
   First, we prove that, for each $k\in \{0,1,\dots,e-1\}$,
\begin{equation}
\label{eqn:bip1}
     \mathsf c_{k}(\operatorname{Char}_{A(\Gamma\oplus \Lambda)}) \equiv
        \begin{cases}
            \mathsf c_{k}(\operatorname{Char}_{A(\Gamma)}) , & \text{ if } k\neq f;\\
            \mathsf c_{k}(\operatorname{Char}_{A(\Gamma)})+2^{e-1-k}, & \text{ if $k = f$}
        \end{cases} \mod{2^{e-k}}.
\end{equation}
    Indeed, since $\operatorname{Char}_{A(\Gamma \oplus \Lambda)}(x)=\operatorname{Char}_{A(\Gamma)}(x)\operatorname{Char}_{A(\Lambda)}(x)$, using (LT1b), we have 
    \begin{align*}
        \mathsf c_{k}(\operatorname{Char}_{A(\Gamma\oplus \Lambda)}) &= \sum_{i=0}^k \mathsf c_{i}(\operatorname{Char}_{A(\Gamma)})\mathsf c_{k-i}(\operatorname{Char}_{A(\Lambda)}) 
        \equiv \mathsf c_{k}(\operatorname{Char}_{A(\Gamma)}) + \mathsf c_{k}(\operatorname{Char}_{A(\Lambda)}) \mod {2^{e-k}}.
    \end{align*}

    Similarly, using Lemma~\ref{lem:sumProdAdjTourn} together with (LT1a), it follows that, for each $k \in \{1,\dots,e-1\}$,
    \begin{equation}
        \label{eqn:oAo1}
        \mathbf{1}^\transpose A(\Gamma \oplus \Lambda)^{k-1}\mathbf{1}\equiv \mathbf{1}^\transpose A(\Gamma)^{k-1}\mathbf{1}\mod{2^{e-k+1}}.
    \end{equation}

    In view of \eqref{eqn:bip1}, it suffices to show that, for each even $k\in \{2,\dots,e -1\}$,
    \begin{equation}
    \label{eqn:coeffDiff}
        \mathsf c_{k}(\operatorname{Char}_{J-2A(\Gamma \oplus \Lambda)}) - \mathsf c_{k}(\operatorname{Char}_{J-2A(\Gamma)}) \equiv 2^{k}\left (\mathsf c_{k}(\operatorname{Char}_{A(\Gamma \oplus \Lambda)})-\mathsf c_{k}(\operatorname{Char}_{A(\Gamma)})\right ) \mod{2^{e}},
    \end{equation}
    which follows from Lemma~\ref{lem:coefficient_link_formula} together with \eqref{eqn:oAo1} and \eqref{eqn:bip1}.
\end{proof}

Now we can prove that the cardinality of $\mathcal C_{e}(\mathsf T_N)$ is at least $2^{\lceil e/2\rceil -2}|\mathcal C_{e-1}(\mathsf T_n)|$ for $e \geqslant 3$ and integers $n$ and $N$, with $N$ large enough relative to $n$.

\begin{lemma}
\label{lem:Tournament_lower_bound}
    Let $n \in \mathbb N$ 
    and $e \geqslant 3$ be an integer.
    Suppose $N = n+\sum_{f=2}^{\lceil e/2\rceil -1}n_{2f}$, where $n_f$ is the order of an $(e,f)$-lift tournament of type I.
    Then $\left |\mathcal C_{e}\left (\mathsf T_N \right ) \right | \geqslant 2^{\lceil e/2\rceil -2}|\mathcal C_{e-1}(\mathsf T_n)|$.
\end{lemma}
\begin{proof}
    By Lemma~\ref{lem:Tt1ffp1} and Lemma~\ref{lem:inductiveTLiftConst}, for each $f \in \{2,\dots,\lceil e/2\rceil -1\}$, there exists an $(e,2f)$-lift tournament $\Lambda_{2f}$ of type I.
    Denote by $n_{f}$ the order of $\Lambda_f$.
    Let $(c_4,c_6,\dots,c_{2\lceil e/2\rceil -2}) \in \{1,\dots,2^e\}^{\lceil e/2\rceil -2}$ such that $\left ( \overline {c_2}^{(e-1)},\overline {c_3}^{(e-1)},\dots,\overline {c_{e-1}}^{(e-1)}\right) \in \mathcal C_{e-1}(\mathsf T_{n})$ for some integers $c_2$ and $c_{2i+1}$ for $i \in \{1,2,\dots,\lfloor e/2\rfloor - 1\}$.
    Let $\Gamma$ be an $n$-vertex tournament such that 
    $\mathsf c_i(\operatorname{Char}_{J-2A(\Gamma)}) \equiv c_i \mod {2^{e-1}}$ for each $i \in \{2,3,\dots,e-1\}$. 
    Thus, for each $i \in \{4,6,\dots,2\lceil e/2\rceil -2\}$, there exists $(d_4,d_6,\dots,d_{2\lceil e/2\rceil -2}) \in \{0,1\}^{\lceil e/2\rceil -2}$ such that $c_i \equiv \mathsf c_i(\operatorname{Char}_{J-2A(\Gamma)}) + 2^{e-1}d_i \mod{2^e}$.
    
    Let 
    \[
    \Gamma^\prime = \Gamma \cup d_{4} \Lambda_{4} \cup d_{6}\Lambda_{6} \cup \dots \cup  d_{2\lceil e/2\rceil -2}\Lambda_{2\lceil e/2\rceil -2}
    \]
    and let $n^\prime$ be the order of $\Gamma^\prime$.
    Then $n^\prime \leqslant n +\sum_{f=2}^{\lceil e/2\rceil -1}n_{2f} =: N$ and (LT1a) implies that $N \equiv n^\prime \equiv n \mod{2^{e}}$. 

    By Lemma~\ref{lem:tournament_lift_graph_effect}, 
    \[
    \mathsf c_i(\operatorname{Char}_{J-2A(\Gamma^\prime)}) \equiv \mathsf c_i(\operatorname{Char}_{J-2A(\Gamma)})+d_i2^{e-1}\equiv c_i \mod{2^e}
    \]
    for each $i\in \{4,6,\dots,2\lceil e/2\rceil -2\}$. 
    Hence, for some integers $c_{2}^\prime, c_3^\prime, \dots, c^\prime_e$ such that $c^\prime_i = c_i$ for each $i\in \{4,6,\dots,2\lceil e/2\rceil -2\}$, we have $\left ( \overline {c^\prime_2}^{(e)},\overline {c^\prime_3}^{(e)},\dots,\overline {c_{e-1}^\prime}^{(e)}, \overline {c^\prime_{e}}^{(e)}\right)\in \mathcal C_{e}(\mathsf T_{n^\prime})$.
    The lemma follows since $\mathcal C_{e}(\mathsf T_{n^\prime}) \subseteq \mathcal C_{e}(\mathsf T_N)$ by Proposition~\ref{prop:lifting_vertices_tournament} and the fact that $N\equiv n^\prime \equiv n \mod{2^{e}}$. 
\end{proof}

Using Lemma~\ref{lem:Tournament_lower_bound}, we can prove the odd case of Theorem~\ref{thm:m3} with a simple induction.

\begin{theorem}
\label{thm:m3odd}
   For each integer $e\geqslant 2$,
    there exists an integer $N_e$ such that $\left |\mathcal C_{e}\left (\mathsf T_n \right ) \right| =  2^{\lfloor \frac{e-2}{2} \rfloor\lfloor \frac{e-3}{2} \rfloor}$ for each odd integer $n \geqslant N_e$.
\end{theorem}
    \begin{proof}
    By Lemma~\ref{lem:Tub}, it suffices to show that $|\mathcal C_{e}(\mathsf T_n)| \geqslant 2^{\lfloor \frac{e-2}{2} \rfloor\lfloor \frac{e-3}{2} \rfloor}$.
         We proceed by induction on $e$.
    For $e \geqslant 4$, the claim is vacuous: one can take $N_e = 1$.

    Now we can assume $e \geqslant 5$ for the inductive step. 
    Take $N_{e}=N_{e-1}+\sum_{f=2}^{\lceil e/2\rceil -1}n_{2f}$ where $n_f$ is the order of an $(e,f)$-lift tournament of type I.
    Suppose $n = N_e + 2k$ for some nonnegative integer $k$.
    Then, by the implicit inductive hypothesis, $|\mathcal C_{e-1}(\mathsf T_{N_{e-1}+2k})| = 2^{\lfloor \frac{e-3}{2} \rfloor\lfloor \frac{e-4}{2} \rfloor}$.
    By Lemma~\ref{lem:Tournament_lower_bound}, we have
    \[
    |\mathcal C_{e}(\mathsf T_n)| \geqslant 2^{\lfloor \frac{e-3}{2} \rfloor\lfloor \frac{e-4}{2} \rfloor} 2^{\lceil e/2\rceil -2} = 2^{\lfloor \frac{e-2}{2} \rfloor\lfloor \frac{e-3}{2} \rfloor}. \qedhere
    \]
    \end{proof}

    The main utility of lift tournaments of type II is the following lemma.

\begin{lemma}\label{lem:tournament_lift_graph_for_odd_index_effect}
Let $e$ be an even integer, $\Gamma$ be a tournament of even order, and $\Lambda$ be an $e$-lift tournament of type II.
Then, for each even $k \in \{2,\dots,e\}$, we have
    \[ \mathsf c_{k}(\operatorname{Char}_{J-2A(\Gamma\oplus \Lambda)}) \equiv
    \begin{cases}
        \mathsf c_{k}(\operatorname{Char}_{J-2A(\Gamma)}), &\text{ if } k \leqslant e-2\\
        \mathsf c_{e}(\operatorname{Char}_{J-2A(\Gamma)})+2^{e-1}, & \text{ if $k = e$}
    \end{cases} \mod{2^{e}}.
    \]
\end{lemma}
\begin{proof}
    Using (LT2c) with the identity $\operatorname{Char}_{A(\Gamma \oplus \Lambda)}(x)=\operatorname{Char}_{A(\Gamma)}(x)\operatorname{Char}_{A(\Lambda)}(x)$, it follows that, for each $k \in \{0,1,\dots,e-2\}$,
    \begin{align}
        \mathsf c_{k}(\operatorname{Char}_{A(\Gamma\oplus \Lambda)}) = \sum_{i=0}^k \mathsf c_{i}(\operatorname{Char}_{A(\Gamma)})\mathsf c_{k-i}(\operatorname{Char}_{A(\Lambda)})
        &\equiv \mathsf c_{k}(\operatorname{Char}_{A(\Gamma)}) \mod {2^{e}} \label{eqn:ck}
    \end{align}

       By Lemma~\ref{lem:sumProdAdjTourn} together with (LT2a), it follows that, for each $i \in \{1,\dots,e-2\}$, we have
    \begin{equation}
        \label{eqn:oAo}
        \mathbf{1}^\transpose A(\Gamma \oplus \Lambda)^{i-1}\mathbf{1}\equiv \mathbf{1}^\transpose A(\Gamma)^{i-1}\mathbf{1}\mod{2^{e}}.
    \end{equation}

    By Lemma~\ref{lem:coefficient_link_formula}, \eqref{eqn:ck},
    and \eqref{eqn:oAo}, for even $k\in \{2,\dots,e-2\}$,
    \begin{align*}
         \mathsf c_{k}(\operatorname{Char}_{J-2A(\Gamma \oplus \Lambda)})&=2^{k} \left (\mathsf c_{k}(\operatorname{Char}_{A(\Gamma\oplus \Lambda)})+\frac{1}{2}\sum_{i=1}^{k}\mathsf c_{k-i}(\operatorname{Char}_{A(\Gamma\oplus \Lambda)})\mathbf{1}^\transpose A(\Gamma \oplus \Lambda)^{i-1}\mathbf{1} \right )\\
         &\equiv 2^{k} \left (\mathsf c_{k}(\operatorname{Char}_{A(\Gamma)})+\frac{1}{2}\sum_{i=1}^{k}\mathsf c_{k-i}(\operatorname{Char}_{A(\Gamma)})\mathbf{1}^\transpose A(\Gamma)^{i-1}\mathbf{1} \right ) \mod{2^{e}}\\
         &\equiv \mathsf c_{k}(\operatorname{Char}_{J-2A(\Gamma)}) \mod{2^{e}}.
    \end{align*}
 
    Similarly, since $\mathsf c_{1}(\operatorname{Char}_{A(\Gamma\oplus \Lambda)}) = \mathsf c_{1}(\operatorname{Char}_{A(\Gamma)}) = 0$ and both  $\mathbf{1}^\transpose A(\Gamma \oplus \Lambda)^{0}\mathbf{1}$ and $\mathbf{1}^\transpose A(\Gamma)^{0}\mathbf{1}$ are even, we have
    \begin{align}
    \label{eqn:ae}
         \mathsf c_{e}(\operatorname{Char}_{J-2A(\Gamma \oplus \Lambda)})
         &\equiv 2^{e-1} \left(\mathbf{1}^\transpose A(\Gamma \oplus \Lambda)^{e-1}\mathbf{1}+\sum_{i=2}^{e-2}\mathsf c_{e-i}(\operatorname{Char}_{A(\Gamma)})\mathbf{1}^\transpose A(\Gamma)^{i-1}\mathbf{1} \right ) \mod{2^{e}};
         \\
\label{eqn:ae2}
\mathsf c_{e}(\operatorname{Char}_{J-2A(\Gamma)})
         &\equiv 2^{e-1} \left(\mathbf{1}^\transpose A(\Gamma)^{e-1}\mathbf{1}+\sum_{i=2}^{e-2}\mathsf c_{e-i}(\operatorname{Char}_{A(\Gamma)})\mathbf{1}^\transpose A(\Gamma)^{i-1}\mathbf{1} \right ) \mod{2^{e}}.
    \end{align}
    Furthermore, using Lemma~\ref{lem:sumProdAdjTourn} together with (LT2a) and (LT2b), we have
    \[
    2^{e-1}\mathbf{1}^\transpose A(\Gamma \oplus \Lambda)^{e-1}\mathbf{1}\equiv 2^{e-1}\mathbf{1}^\transpose A(\Gamma)^{e-1}\mathbf{1}+2^{e-1} \mod{2^e} 
    \]
    Combining with \eqref{eqn:ae} and \eqref{eqn:ae2}, yields the congruence $\mathsf c_{e}(\operatorname{Char}_{J-2A(\Gamma \oplus \Lambda)}) \equiv \mathsf c_{e}(\operatorname{Char}_{J-2A(\Gamma)}) + 2^{e-1} \mod{2^e}$.
\end{proof}

Now we can prove that the cardinality of $\mathcal C_{e}(\mathsf T_N)$ is at least $2^{e/2 -1}|\mathcal C_{e-1}(\mathsf T_n)|$ for even integers $e \geqslant 6$ and $n$ and $N$ are both even, with $N$ large enough relative to $n$.

\begin{lemma}
\label{lem:2em3belowT1}
    Let $n \in \mathbb N$ be even and $e \geqslant 6$ be an even integer.
    Suppose $N = n+\sum_{f=2}^{e/2 -1} n_{2f} + m$
    where $n_f$ is the order of an $(e,f)$-lift tournament of type I, and $m$ is the order of an $e$-lift tournament graph type II.
    Then $\left |\mathcal C_{e}\left (\mathsf T_N \right ) \right | \geqslant 2^{e/2 -1}|\mathcal C_{e-1}(\mathsf T_n)|$.
\end{lemma}
\begin{proof}
    By Lemma~\ref{lem:Tt1ffp1} and Lemma~\ref{lem:inductiveTLiftConst}, for each $f \in \{2,\dots,e/2 -1\}$, there exists an $(e,2f)$-lift tournament of type I $\Lambda_{2f}$ and, by Lemma~\ref{lem:Tt2}, there exists an $e$-lift tournament of type II $\Lambda_e$.
    Denote by $n_f$ the order of each tournament $\Lambda_f$, and $m$ the order of the tournament $\Lambda_e$.
    Note that $n_f= \mathbf 1^\transpose A(\Lambda_f)^{0} \mathbf 1$ and $m= \mathbf 1^\transpose A(\Lambda_e)^{0} \mathbf 1$.
    Hence, (LT1a) and (LT2a) imply that $2^e$ divides both $n_f$ and $m$.
    
    Let $(c_4,c_6,\dots,c_{e-2}) \in \{1,\dots,2^e\}^{e/2 -2}$ such that $\left ( \overline {c_2}^{(e-1)},\overline {c_3}^{(e-1)},\dots,\overline {c_{e-1}}^{(e-1)}\right) \in \mathcal C_{e-1}(\mathsf T_{n})$.
    Let $\Gamma$ be an $n$-vertex graph such that 
 $\mathsf c_i(\operatorname{Char}_{J-2A(\Gamma)}) \equiv c_i \mod {2^{e-1}}$ for each $i \in \{2,3,\dots,e-1\}$. 
 Furthermore, by Lemma~\ref{lem:first3coeffs}, we have  $\mathsf c_e(\operatorname{Char}_{J-2A(\Gamma)}) \equiv 0 \mod {2^{e-1}}$.
 Accordingly, let $c_e \in \{0,2^{e-1}\}$.
 Thus, there exists $(d_4,d_6\dots,d_{e}) \in \{0,1\}^{e/2-1}$ such that $c_{2i} \equiv \mathsf c_{2i}(\operatorname{Char}_{J-2A(\Gamma)}) + 2^{e-1}d_{2i} \mod{2^e}$ for each $i \in \{2,3,\dots,e/2\}$.
    Let 
    \[
    \Gamma^\prime = \Gamma \cup d_{4} \Lambda_{4} \cup d_{6}\Lambda_{6} \cup \dots \cup  d_{e -2}\Lambda_{e -2}\cup d_{e} \Lambda_e
    \]
    and let $n^\prime$ be the order of $\Gamma^\prime$.
    Then $n^\prime \leqslant n +\sum_{f=2}^{e/2-1}n_{2f}+m =: N$ and (LT1a) and (LT2a) imply that $N \equiv n^\prime \equiv n \mod{2^{e}}$. 
    Then, by Lemma~\ref{lem:tournament_lift_graph_effect} and Lemma~\ref{lem:tournament_lift_graph_for_odd_index_effect}, we have 
    \[
    \mathsf c_{2i}(\operatorname{Char}_{J-2A(\Gamma^\prime)}) \equiv \mathsf c_{2i}(\operatorname{Char}_{J-2A(\Gamma)})+d_{2i}2^{e-1}\equiv c_{2i} \mod{2^e}
    \]
    for $i\in \{2,3,\dots,e/2\}$. 
     Hence, for some integers $c_{2}^\prime, c_3^\prime, \dots, c^\prime_e$ such that $c^\prime_{2i} = c_{2i}$ for each $i\in \{2,3,\dots,e/2\}$, we have $\left ( \overline {c^\prime_2}^{(e)},\overline {c^\prime_3}^{(e)},\dots,\overline {c_{e-1}^\prime}^{(e)}, \overline {c^\prime_{e}}^{(e)}\right)\in \mathcal C_{e}(\mathsf T_{n^\prime})$.
    The lemma follows since $\mathcal C_{e}(\mathsf T_{n^\prime}) \subseteq \mathcal C_{e}(\mathsf T_N)$ by Proposition~\ref{prop:lifting_vertices_tournament} and the fact that $N\equiv n^\prime \equiv n \mod{2^{e}}$. 
\end{proof}

Using Lemma~\ref{lem:Tournament_lower_bound} and Lemma~\ref{lem:2em3belowT1}, we can prove the even counterpart of Theorem~\ref{thm:m3odd}, thereby completing the proof of Theorem~\ref{thm:m3}.

\begin{theorem}
\label{thm:m3even}
     For each integer $e\geqslant 4$,
    there exists an integer $N_e$ such that $\left |\mathcal C_{e}\left (\mathsf T_n \right ) \right| =  2^{\lfloor \frac{e-1}{2} \rfloor\lfloor \frac{e-2}{2} \rfloor}$ for each even integer $n \geqslant N_e$.
\end{theorem}
\begin{proof}
By Lemma~\ref{lem:Tub}, it suffices to show that $|\mathcal C_{e}(\mathsf T_n)| \geqslant 2^{\lfloor \frac{e-1}{2} \rfloor\lfloor \frac{e-2}{2} \rfloor}$.
    We proceed by induction on $e$.
    For $e = 4$, we can take $N_4 = 4$ and we prove that $\left |\mathcal C_{4}\left (\mathsf T_n \right ) \right|= 2$ for each even $n \geqslant N_4$.
    
    Let $V_n = \{0,\dots,n-1\}$.
    Define the tournament $\Gamma_1 =(V_n,E_1)$ such that $(i,j) \in E_1$ if and only if $i < j$.
    Define the tournament $\Gamma_2 =(V_n,E_2)$ such that $(i,j) \in E_2$ if and only if $i < j$ except for $i = n-3$ and $j = n-1$ where $(n-1,n-3) \in E_2$.
    It is clear that $\mathbf 1^\transpose A(\Gamma_1)^{3}\mathbf 1= {n\choose{4}}$.
    Indeed, each $4$-walk $(v_1,v_2,v_3,v_4)$ in $\Gamma_1$ must have $v_1<v_2<v_3<v_4$. 
    For a $4$-walk $(v_1,v_2,v_3,v_4)$ in $\Gamma_2$, compared to the $4$-walks in $\Gamma_1$, we lose the $4$-walks in $\Gamma_1$ with $(v_1,v_2,n-3,n-1)$ where $v_1,v_2 \in \{1,\dots,n-4\}$ and we gain $(n-1,n-3,n-2,n-1)$,
    $(v_1,n-1,n-3,n-2)$ where $v_1 \in \{1,\dots,n-4\}\cup \{n-2\}$, and $(v_1,v_2,n-1,n-3)$ where $v_2 = n-2$ and $v_1 \in \{1,\dots,n-3\}$ or $v_1,v_2 \in \{1,\dots,n-4\}$.
    Hence, $\mathbf 1^\transpose A(\Gamma_2)^{3}\mathbf 1= {n\choose{4}}+2n-3$.
    We claim that, for each $i \in \{1,2\}$,
    \[
    \mathsf c_4(\operatorname{Char}_{J-2A(\Gamma_i)}) \equiv 8 \mathbf 1^\transpose A(\Gamma_i)^3 \mathbf 1 \mod {16}.
    \]
    Indeed, apply Lemma~\ref{lem:coefficient_link_formula}, together with
 $\mathsf c_1(\operatorname{Char}_{A(\Gamma_i)})=\mathsf c_2(\operatorname{Char}_{A(\Gamma_i)})=0$ and $\mathbf 1^\transpose A(\Gamma_i)^{0}\mathbf 1$ is even. 
    Hence, $\left |\mathcal C_{4}\left (\mathsf T_n \right ) \right|= 2$.

      Now we can assume $e \geqslant 5$ for the inductive step. 
    Take 
    \[
    N_{e}=N_{e-1}+\sum_{f=2}^{\lceil e/2\rceil -1} n_{2f}+ \begin{cases}
        m_e & \text{ if $e$ is even; } \\
        0 & \text{ if $e$ is odd, } \\
    \end{cases}
    \]
    where $n_f$ is the order of $(e,2f)$-lift tournament of type I and $m_e$ the order of an $e$-lift tournament of type II.
    Let $n \geqslant N_e$ be an even integer.
    Suppose $n = N_e + 2k$ for some nonnegative integer $k$.
    Then, by the inductive hypothesis, $|\mathcal C_{e-1}(\mathsf T_{N_{e-1}+2k})| = 2^{\lfloor \frac{e-2}{2} \rfloor\lfloor \frac{e-3}{2} \rfloor}$.
    By Lemma~\ref{lem:Tournament_lower_bound} or Lemma~\ref{lem:2em3belowT1}, we have
    \[
    |\mathcal C_{e}(\mathsf T_n)| \geqslant 2^{\lfloor \frac{e-3}{2} \rfloor\lfloor \frac{e-2}{2} \rfloor} 2^{\lfloor  e/2\rfloor -1} = 2^{\lfloor \frac{e-2}{2} \rfloor\lfloor \frac{e-1}{2} \rfloor},
    \]
    as required.
\end{proof}

\bibliographystyle{plain}

\bibliography{bib}

\end{document}